%% file: MeFe2019_arXiv.tex
\documentclass[english,reqno]{amsart}
\usepackage[margin=2.5cm]{geometry}      

\input{preamble}
\input{notation}

\begin{document}


%
%

\title{Equilibria of an aggregation model with linear diffusion in domains with boundaries}

\author{Daniel A. Messenger}
\address{Department of Mathematics, Simon Fraser University, 8888 University Dr., Burnaby, BC V5A 1S6, Canada.}
\email{daniel.messenger@colorado.edu\footnote{D.M.'s current affiliation: Department of Applied Mathematics, University of Colorado Boulder, 11 Engineering Dr., Boulder, CO 80309, USA.}} 

\author{Razvan C. Fetecau}  
\address{Department of Mathematics, Simon Fraser University, 8888 University Dr., Burnaby, BC V5A 1S6, Canada.}
\email{van@math.sfu.ca}

\maketitle


\begin{abstract}
We investigate the effect of linear diffusion and interactions with the domain boundary on swarm equilibria by analyzing critical points of the associated energy functional. Through this process we uncover two properties of energy minimization that depend explicitly on the spatial domain: (i) unboundedness from below of the energy due to an imbalance between diffusive and aggregative forces depends explicitly on a certain volume filling property of the domain, and (ii) metastable mass translation occurs in domains without sufficient symmetry. From the first property, we present a sharp condition for existence (resp. non-existence) of global minimizers in a large class of domains, analogous to results in free space, and from the second property, we identify that external forces are necessary to confine the swarm and grant existence of global minimizers in general domains. We also introduce a numerical method for computing critical points of the energy and give examples to motivate further research.
\end{abstract}




\section{Introduction}
\label{sect:intro}

We consider minimizers of the following nonlocal and non-convex energy functional: 
\begin{align}
\label{eqn:energy}
\CalE^\nu[\mu] &= \frac{1}{2}\int_D\int_D K(x-y)\,d\mu(x)\,d\mu(y)+\nu\int_D\log(\rho(x))\,d\mu(x)+\int_DV(x)\,d\mu(x),
\vspace{0.5cm} 
\end{align}
for measures $\mu$ that are absolutely continuous with respect to the Lebesgue measure ($\rho$ denotes the density of $\mu$) and for general domains $D\subset \R^d$ with smooth boundary. Here, $K$ and $V$ represent interaction and external potentials, respectively, and $\nu>0$ is the diffusion parameter.

Minimizers of the energy $\CalE^\nu$ relate to equilibria of the aggregation model with linear diffusion,
\begin{equation}\label{aggdD}
	\begin{dcases} 
		\ppt \munu + \nabla \cdot \left(\munu \,v^\nu \right) = \nu\Delta \munu, & x\in D\\[8pt]
\Big\langle n(x),\, \rhonu(x) \, v^\nu(x)+\nu\nabla \rhonu(x)\Big\rangle = 0, & x\in \partial D\\[8pt] 
		\munu\big\vert_{t=0} = \mu^\nu_0 \in \CalP_2(D),
	\end{dcases}\\[2pt]
\end{equation}
where $\rhonu$ is the density of $\munu$, $v^\nu$ denotes the swarm velocity
\begin{equation}\label{aggdvelocity}
v^\nu(x) := -\int_D \nabla K(x-y)\dd\munu(y)-\nabla V(x),
\end{equation}
and $n(x)$ denotes the unit outward normal to $\partial D$ at $x$. Specifically, weak-measure solutions of \eqref{aggdD} are 2-Wasserstein gradient flows of $\CalE^\nu$ on the space $\CalP_2(D)$ of probability measures with finite second moment \cite{ambrosio2008gradient,CaMcVi2006}, such that steady states of \eqref{aggdD} correspond to critical points of $\CalE^\nu$. 

Aggregation-diffusion models of type \eqref{aggdD} and their associated energies of the form \eqref{eqn:energy} (with or without diffusion) appear in many phenomena, including  biological swarms \cite{MoKe1999,LeToBe2009},  granular media \cite{carrillo2003kinetic}, self-assembly of nanoparticles \cite{HoPu2006} and opinion dynamics \cite{MotschTadmor2014}. In such applications the linear diffusion can model anti-crowding, locally repulsive interactions, or it can be the result of Brownian noise included in the model. On the other hand, the potential $K$ models nonlocal social interactions such as attraction and repulsion between the members of a group (e.g., individuals of a biological swarm). There has been extensive research on various aspects of this model,  with the vast majority of these works being concerned with the model set up in free space ($D=\Rd$). We will briefly review below some of this literature.

In the absence of diffusion, the study of the evolution model in free space has been a very active area of research recently \cite{BertozziCarilloLaurent,BertozziLaurent,LeToBe2009,FeRa10,carrillo2011global}. The behaviour of its solutions relies fundamentally on the nature and properties of the interaction potential $K$. Consequently, attractive potentials lead to (finite or infinite-time) blow-up \cite{BertozziCarilloLaurent,HuBe2010}, while balancing attraction and repulsion can generate finite-size, confined aggregations \cite{FeHuKo11,LeToBe2009}. The model with diffusion has an extensive literature of its own; we refer here to \cite{carrillo2003kinetic,CaMcVi2006} for comprehensive studies on well-posedness of solutions to this model using the theory of gradient flows in probability spaces \cite{ambrosio2008gradient}, and to \cite{HoPu2006,HuFe2013} for studies on equilibria of the diffusive model with applications to aggregation/collective behaviour.

Critical points of the interaction energy \eqref{eqn:energy} (or equivalently, equilibria of the dynamic model) have been studied in various papers recently. For the case $\nu =0$, existence of global minimizers has been established in \cite{CaChHu2014,CaCaPa2015,ChFeTo2015,SiSlTo2015}, while qualitative properties such as dimensionality, size of the support, symmetry and stability have been investigated in \cite{balague2013dimensionality,BaCaLaRa2013,LeToBe2009,FeHuKo11,FeHu13}. A provoking gallery of such minimizers is presented for instance in \cite{KoSuUmBe2011}; it contains aggregations on disks, annuli, rings, soccer balls, and others. In the presence of diffusion ($\nu>0$), the focus is the competition between local repulsion effects (diffusion) and nonlocal attractive interactions that provides existence (or lack of) energy minimizers. This delicate balance of such forces was recently investigated by Carillo et al. in \cite{carrillo2018existence} for the case of free space; this work is central to our paper and we will return to it frequently throughout.  

Our main interest in this paper lies in domains with boundaries, which are very relevant to many realistic physical settings (e.g., the boundary may be an obstacle in the environment, such as a river or the ground; the latter arises for instance in the locust model from \cite{ToDoKeBe2012}). Equilibria for the aggregation model without diffusion in domains with boundaries have been studied in \cite{bernoff2011primer}, while the well-posedness of its solutions (in the probability measure space) has been established in \cite{carrillo2014nonlocal,wu2015nonlocal}. Also, in a recent study \cite{fetecau2017swarm}, the authors identified a flaw of the aggregation equation with zero diffusion in domains with boundaries: its solutions can evolve into unstable equilibria. This is a surprising degeneracy of model \eqref{aggdD} without diffusion, given that it has a gradient flow formulation. From this perspective, adding diffusion can be seen as a regularizing mechanism  \cite{evers2016metastable,zhang2017continuity,fetecau2018zero}.

We also mention that there has been extensive work on aggregation models with repulsive effects modelled by nonlinear diffusion (see for instance \cite{BeRoBe2011,CaHiVoYa2019} and references therein). The two modes of diffusion (linear vs. nonlinear) result in different features of equilibria/minimizers of the associated energy. In particular, nonlinear diffusion models admit compactly supported equilibria\cite{BurgerDiFrancescoFranek,BuFeHu14,CaHiVoYa2019}, in contrast with equilibria for linear diffusion which can only have full support within the domain (see Section \ref{sect:critical-points} of the present paper). Another important distinction is that the linear diffusion model has an underlying particle system associated to it (modelled by Brownian motion), while nonlinear diffusion does not, making the former more flexible in terms of numerical simulations \cite{fetecau2018zero}. 

We discuss now the organization and the contributions of the present paper. Notation and assumptions are introduced in Section \ref{sect:prelim}. In Section \ref{sect:critical-points} we establish some useful properties of critical points of the energy $\CalE^\nu$, including a new energy argument showing that minimizers $\gmin$ satisfy $\supp{\gmin}=D$. Section \ref{sect:free-space} contains a brief review of various results in free space established in \cite{carrillo2018existence}, while Sections \ref{sect:non-exist} and \ref{sect:existence} present our results on energy minimization in domains with boundaries. 


In Section \ref{sect:non-exist}, we identify two primary phenomena which lead to non-existence of global minimizers of $\CalE^\nu$ and depend on properties of the spatial domain. The first non-existence phenomenon (discussed in Section \ref{subsect:imbalance}) involves an imbalance of diffusive and aggregative forces and is analogous to non-existence results for free space found in \cite{carrillo2018existence}. An interesting novelty of our study, however, is that we find that non-existence as a result of diffusion-dominated spreading depends on the \textit{effective volume dimension} of the domain (defined in Section \ref{sect:prelim}) which informally describes the number of orthogonal directions in the domain which independently extend to infinity.

The second non-existence phenomenon (presented in Section \ref{subsect:escape}) is a result of asymmetries within $D$ and involves metastable translation of the swarm under diffusion-mediated repulsion from the boundary $\partial D$. This phenomenon rules out the existence of minimizers in the absence of an external potential ($V=0$) for large classes of unbounded domains. Specifically, the necessary condition in Theorem \ref{domassymthm} for critical points implies that if the domain is not suitably symmetric, arbitrarily large attraction at infinity (e.g., an attractive potential such as $K(x) = |x|^p/p$ with $p$ large) cannot contain the swarm regardless of how small the diffusion. Non-existence of this type necessitates the use of an external potential $V$ to confine the swarm, in contrast to results in free space, where existence of minimizers for $V=0$ is guaranteed for sufficiently strong attraction at infinity \cite{carrillo2018existence}. 

Section \ref{sect:existence} is devoted to the existence of global minimizers in light of the phenomena discussed in Section \ref{sect:non-exist}. Theorem \ref{tightbound} establishes a sharp condition for existence in domains of type $D = F\times \RR^{d-m}$ with an effective volume dimension of $d-m$, where $F\subset \RR^m$ is a compact $m$-dimensional set with $0<m<d$. This result serves as a generalization of the sharp existence condition in free space established in \cite{carrillo2018existence}. Theorems \ref{existence1} and \ref{existence2} then provide sufficient conditions for existence in general domains with boundaries, which involve establishing a minimal set of requirements on $V$ for confinement, in light of the metastable translations discussed in Section \ref{subsect:escape} that occur when $V=0$. 

Finally, in Section \ref{sect:numerics} we discuss some of the findings in the previous sections through numerical examples of critical points computed using a fixed-point iteration scheme with relaxation.  We illustrate both purely attractive and attractive-repulsive interaction potentials, with the latter featuring non-uniqueness of critical points.


\section{Preliminaries and Assumptions}
\label{sect:prelim}

In this section we provide some preliminaries and background, as well as list the assumptions we make on the potentials and the domain.
\subsection*{Notations}
Let $\CalB^d$ denote the Borel $\sigma$-algebra on $\Rd$.
For $A\in \CalB^d$, $|A|$ denotes the volume of $A$ (with respect to the Lebesgue measure) and $\ind{A}$ represents the indicator function of $A$. Let $B_R(x)$ denote the $d$-dimensional Euclidean ball of radius $R$ centred at $x\in\Rd$. For $D\in \CalB^d$, let $\CalP(D)$ denote the set of Borel probability measures on $D$ and $\CalP^{ac}(D)\subset\CalP(D)$ the set of absolutely continuous measures with respect to the Lebesgue measure.
%
%
%
For $\mu\in \CalP(D)$, $x\in \Rd$ and $\gamma\in \RR$, define the $\gamma^{th}$ moment of $\mu$ centred at $x$ as
\begin{equation}\label{moment}
M^x_\gamma(\mu) := \int_D|y-x|^\gamma \dd\mu(y),
\end{equation}
with 
\begin{equation}
\label{moment0}
M_\gamma(\mu):=M^0_\gamma(\mu),
\end{equation}
and the centre of mass of $\mu$ by 
\begin{equation}\label{centreofmass}
\com\left(\mu\right):=\int_D x\,d\mu(x).
\end{equation}
\begin{rmrk}
\normalfont
Throughout, we will often refer to an absolutely continuous measure directly by its density $\rho$, and by abuse of notation sometimes write $\rho\in \CalP^{ac}(D)$ to mean $d\rho(x) = \rho(x)\,dx$. 
\end{rmrk}
\subsection*{Weak-* Relative Compactness and Tightness}
\hspace{5mm} We say that a sequence $\left\{\mu_n\right\}_{n\geq 0} \subset \CalP(D)$ converges weakly-* to $\mu\in \CalP(D)$ and write $\mu_n\overset{*}{\rightharpoonup} \mu$ if for every bounded continuous function $f:D\to \RR$ we have 
\[\lim_{n\to \infty}\int_Df(x)\,d\mu_n(x)  = \int_Df(x)\,d\mu(x).\]
A collection of measures $\CalF\subset \CalP(D)$ is said to be weakly-* relatively compact if for every sequence $\left\{\mu_n\right\}_{n\geq 0} \subset \CalF$ there exists a subsequence $\left\{\mu_{n_k}\right\}_{k\geq 0}$ which converges weakly-* to some $\mu\in \CalP(D)$. A collection of measures $\CalF\subset \CalP(D)$ is said to be \textit{tight} if for every $\epsilon>0$ there exists a compact set $K_\epsilon\subset D$ such that 
\[\mu(K_\epsilon^c)<\epsilon\] 
for every $\mu\in \CalF$, where $K_\epsilon^c$ denotes the complement of $K_\epsilon$ within $D$ (i.e. $K_\epsilon^c = D\setminus K_\epsilon$). Recall that weak-* relative compactness and tightness are related by Prokhorov's Theorem:
\begin{lemm}{(Prokhorov's theorem \cite[Chapter 1, Section 5]{billingsley2013convergence})}
A collection of measures $\CalF\subset \CalP(D)$ is weakly-* relatively compact if and only if it is tight.
\end{lemm}


\subsection*{The $p$-Wasserstein Space}

\noindent For $p\in [1,\infty)$, define the space  
\begin{equation}\label{p1d}
\mathcal{P}_p(D):=\left\{\mu\in\mathcal{P}(D):~M_p(\mu)<+\infty\right\},
\end{equation}
where $M_p(\mu)$ is defined \eqref{moment0}. The $p$-Wasserstein distance on $\mathcal{P}_p(D)$ is then
\begin{equation}\label{wass}
\CalW_p(\mu,\eta)=\left(\inf_{\pi\in\Lambda(\mu,~\eta)}\Big\{\int_{D\times D}|x-y|^p\dd\pi(x,y)\Big\}\right)^{\frac{1}{p}}=\left(\inf_{X\sim \mu,Y\sim \eta}\Big\{\EE[|X-Y|^p]\Big\}\right)^{\frac{1}{p}},
\end{equation}
where $\Lambda(\mu,~\eta)$ is the set of joint probability measures on ${D}\times{D}$ with marginals $\mu$ and $\eta$, also known as \textit{transport plans}, and $(X,Y)$ ranges over all possible couplings of random variables $X$ and $Y$ with laws $\mu$ and $\eta$, respectively.

Recall that for each $p\in[1,\infty)$ the metric space $(\CalP_p(D),\CalW_p)$ is complete and convergence in $(\CalP_p(D),\CalW_p)$ is equivalent to weak-* convergence of measures. We also have the following useful upper bound on $\CalW_p(\mu,\eta)$ in terms of the total variation measure $d|\mu-\eta|$:
\begin{lemm}\label{ELlemm2}
\cite[Ch. 6]{villani2008optimal}
For all $p\in [1,\infty)$, any $x_0 \in D$ and $\mu,\eta \in \CalP_p(D)$,
\[\CalW_p^p(\mu,\eta)\leq 2^{(p-1)}\int_D|x-x_0|^pd|\mu-\eta|(x).\]
\end{lemm}
\noindent We refer readers to the books \cite{ambrosio2008gradient,villani2008optimal} for further background on $p$-Wasserstein spaces.
%
%
\subsection*{Associated Energy}

The energy functional \eqref{eqn:energy} can be written as:
\begin{equation}
\label{eqn:energy-3parts}
\CalE^\nu[\mu] = \CalK[\mu] + \nu \CalS[\mu] + \CalV[\mu],
\end{equation}
where $\CalK$, $\CalS$ and $\CalV$ are referred to as the interaction energy, entropy and potential energy, respectively, and are defined for all $\mu\in \CalP(D)$ by:
\begin{equation}\label{energyK}
\CalK[\mu] := \frac{1}{2}\int_D\int_D K(x-y)\dd\mu(y)\dd\mu(x),
\end{equation}
\begin{equation}\label{energyS}
\CalS[\mu] := \begin{dcases} \int_D\rho(x)\log(\rho(x))\dd x, &\quad \text {if }\mu\in\CalP^{ac}(D)  \text{ with } d\mu(x) = \rho(x)\dd x,\\
+\infty, & \quad \text{otherwise,}
\end{dcases}
\end{equation}
and
\begin{equation}
\CalV[\mu] := \int_DV(x)\dd\mu(x).
\end{equation}
In this way, the energy $\CalE^\nu$ is defined on the entire space $\CalP(D)$, but takes the value $+\infty$ on measures which are not absolutely continuous.

Extrema of $\CalE^\nu$ are defined as in \cite{carrillo2018existence}: for $r>0$ we define a $\CalW_p$-$r$ local minimizer of $\CalE^\nu$ to be a measure $\gmin$ such that 
\begin{equation}\label{wprmin}
\CalE^\nu[\gmin]\leq \CalE^\nu[\eta] \txt{0.5}{for all} \eta \in \ballp{p}{\gmin}{r},
\end{equation}
where $\ballp{p}{\gmin}{r}$ is the ball of radius $r$ in $\CalP_p(D)$ centred at $\gmin$. A $\CalW_p$-$r$ local maximizer is defined analogously by reversing the inequality. In what follows, we will use the terms \textit{minimizer}, \textit{maximizer} or \textit{extremizer} in reference to condition \eqref{wprmin} whenever the $\CalW_p$ metric has been established.

\begin{rmrk}\label{ac-gmin}
\normalfont
In light of the definition of the entropy \eqref{energyS}, whenever a global minimizer $\gmin$ of $\CalE^\nu$ exists, it follows that $\gmin\in \CalP^{ac}(D)$, as one can always find a measure $\mu\in \CalP^{ac}(D)$ for which $\CalE^\nu[\mu]<+\infty$ (e.g. for a compact, $d$-dimensional set $F\subset D$, the energy $\CalE^\nu\left[\frac{1}{|F|}\ind{F}\right]$ is finite). 
\end{rmrk}


\subsection*{Assumptions}

\hspace{5mm} Throughout we will make the following minimal assumptions about potentials $K$ and $V$, and the spatial domain $D$.
\noindent
\begin{assum}[Potentials]\label{assumKV}
\normalfont
\quad
\begin{enumerate}[label=(\roman*)]
\item (Local integrability) $K,V\in L^1_{loc}(\Rd)$.
\item (Lower semicontinuity) $K$ and $V$ are lower semicontinuous.
\item (Symmetry of $K$) $K(x) = K(-x)$ for all $x\in \mathbb{R}^d$.
%
%
\end{enumerate}
\end{assum}
\begin{assum}[Domain]\label{assumD}
\normalfont
\quad
\begin{enumerate}[label=(\roman*)]
\item (Domain topology) $D\in \CalB^d$ is closed, connected, and satisfies $|D|>0$. 
\item (Boundary regularity) There exists a unique outward normal vector $n(x)$ associated to almost every $x\in \partial D$.
\end{enumerate}
\end{assum}
%
%
\subsection*{Effective Volume Dimension} 

As we will show in Theorems \ref{NE} and \ref{tightbound}, the following property is a crucial component of the asymptotic upper bound on $K$ below which diffusion dominates and infinite spreading occurs, leading to non-existence of global minimizers of $\CalE^\nu$. Moreover, in domains of the form \eqref{Fdomain} below, this asymptotic bound is sharp.
 
For $D \in \CalB^d$ define the function
\begin{equation}\label{fDr}
V_{_D}(r) = \sup_{x\in D}\left\vert D\cap B_r(x)\right\vert.\\[-3pt] 
\end{equation}
We then define the \textit{effective volume dimension} $\eff$ of $D$ by
\begin{equation}\label{effvoldim}
\eff\ :=\ \sup\Bigl\{s \in \RR \txt{0.2}{:} V_{_D}(r) \ \gtrsim\  r^s \txt{0.2}{as} r \to \infty\, \Bigr\}.
\end{equation}
%
%
Consequently, for some $C>0$ and $r'>0$,
\begin{equation}\label{effvolbound}
V_{_D}(r)\ \geq\  C\,r^{\eff} \txt{0.4}{for all} r>r'\,;
\end{equation}
moreover, $\eff$ is the largest value such that \eqref{effvolbound} holds. In words, the largest volume of $D$ intersect a ball of radius $r$ grows proportionally to $r^{\eff}$. In this way, $\eff\in[0,d]$ and $\eff = 0$ if $D$ is bounded. For intuition in three dimensions, example domains with $\eff = 1$ and $\eff=2$ are an open right cylinder and the space between two infinite parallel planes, respectively. 

In Theorem \ref{tightbound} we consider domains of the form 
\begin{equation}\label{Fdomain}
D = F\times \RR^{d-m},
\end{equation}
where $F$ is a compact subset of $\RR^m$ for some $m\in\left\{1,\dots,d-1\right\}$ and satisfies $|F|>0$. In this case, $D$ has effective volume dimension $\eff = d-m$. To see this, note that
\[V_D(r)\leq |F|(2r)^{d-m},\]
and letting $H^{d-m}_a$ denote a $(d-m)$-dimensional hypercube of side-length $a$, for $r>\frac{\sqrt{d}}{2}\text{diam}(F)$ we have 
\[V_D(r)\geq |F\times H^{d-m}_{\frac{2}{\sqrt{d}}r}| = |F|\left(\frac{2}{\sqrt{d}}r\right)^{d-m}.\]
Together this implies $V_{_D}(r) \sim Cr^{d-m}$ as $r\to\infty$. In this way, $\eff$ is exactly the number of orthogonal directions in $D$ which independently extend to infinity. 


\section{Critical Points of the Energy}
\label{sect:critical-points}
Interaction energies of type \eqref{eqn:energy} (with or without linear/nonlinear diffusion) have been studied extensively in free space ($D = \Rd$). It is well known that critical points of the energy $\CalE^\nu$ in free space satisfy the Euler-Lagrange equation \eqref{ELcond} given below. This equation had been derived in various papers, we refer for instance to \cite{balague2013dimensionality} for a derivation using the $\CalW_2$ metric for $\CalE^\nu$ without diffusion, and to \cite{carrillo2018existence} for interaction energies  with general diffusion under the $\CalW_\infty$ metric. We include a derivation of \eqref{ELcond} here under general $\CalW_p$ metrics over general domains $D$, for lack of a direct reference, using the techniques in \cite{balague2013dimensionality} and  \cite{carrillo2018existence}. 

First, we highlight in Theorem \ref{suppmin} the property that minimizers of $\CalE^\nu$ for every $\nu>0$ are supported on the whole domain $D$. This is briefly mentioned in \cite{carrillo2018existence} for free space and is justified by the authors using the Euler-Lagrange equation. As there is no reason a priori why in general domains $D$ a minimizer $\gmin$ should simultaneously satisfy the Euler-Lagrange equation and have $\supp{\gmin}=D$, we present a proof of Theorem \ref{suppmin} without the Euler-Lagrange equation.
%
%
%
\begin{thm}\label{suppmin}[Minimizers have full support]
Let Assumptions \ref{assumKV} and \ref{assumD} hold. In addition, assume that there exist constants $p_K \geq 0$, $R_K>0$, and $C_K>0$ such that
\begin{equation}\label{powerlaw}
0 \leq K(x) \leq C_K|x|^{p_K} \txt{0.5}{ for all } |x| > R_K.
\end{equation}
If $\critrho\in \CalP^{ac}_p(D)$ is a $\CalW_p$-$r$ local minimizer of $\CalE^\nu$ where $p\in \big[\max\left\{1,p_{_K}\right\}, \infty\big]$, then $\supp{\critrho} = D$.
\end{thm}
\begin{proof}
We proceed by contradiction. Let $\critrho$ be a $\CalW_p$-$r$ local minimizer and assume to the contrary that $\supp{\critrho} \subsetneq D$. By definition, $\supp{\critrho}$ is a closed subset of $D$, hence $D\,\setminus\,\supp{\critrho}$ must have positive Lebesgue measure. This implies that there exists a point $x_0 \in\partial\, \supp{\critrho}$ and $\delta>0$ such that the set $A = B_\delta(x_0)\cap \Big(D\setminus\supp{\critrho}\Big)$ has positive Lebesgue measure $|A|$. We now construct a measure $\eta$ with $\CalW_p(\critrho,\eta)<r$ such that $\CalE^\nu[\eta]<\CalE^\nu[\critrho]$, contradicting the assumption that $\critrho$ is a $\CalW_p$-$r$ local minimizer of $\CalE^\nu$. Define
\[\eta = (1-\alpha)\critrho +\alpha\frac{1}{|A|}\ind{A},\]
where $\alpha\in (0,1)$ will be picked in two stages.

First, using Lemma \ref{ELlemm2} we have 
\begin{align*}
\CalW^p_p(\critrho,\eta)&\leq 2^{(p-1)}\int_D|x-x_0|^p\,d|\critrho-\eta|(x)\\ 
&=\alpha\,2^{(p-1)}\left(\int_D|x-x_0|^p\,d\critrho(x)+\frac{1}{|A|}\int_A|x-x_0|^p\,dx\right)\\
&\leq \alpha\,2^{(p-1)}\left(M^{x_0}_p(\critrho)+\delta^p\right),
\end{align*}
and so we choose
\begin{equation}\label{alphbound1}
\alpha<\min\left\{1,\ \dfrac{r^p}{2^{(p-1)}\left(M^{x_0}_p(\critrho)+\delta^p\right)}\right\}
\end{equation}
to ensure that $\CalW_p(\critrho,\eta) < r$.

Next, we find an additional constraint on $\alpha$ to ensure that $\CalE^\nu[\eta]<\CalE^\nu[\critrho]$ by bounding terms in the energy. For any $x\in \RR$ we have $(1-\alpha)^2x < x+2\alpha|x|$, and so a direct calculation of the interaction energy yields the bound
\begin{align*}
\CalK[\eta] &= (1-\alpha)^2\CalK[\critrho] + \frac{\alpha(1-\alpha)}{|A|}\int_D\left(\int_A K(x-y)dy\right)\critrho(x)\,dx+\frac{\alpha^2}{2|A|^2}\int_A\int_A K(x-y)\,dxdy\\
&< \CalK[\critrho]+\alpha\left(2\left\vert\CalK[\critrho]\right\vert+\frac{1}{|A|}\underbrace{\int_D\left(\int_A \left\vert K(x-y)\right\vert\,dy\right)\critrho(x)\,dx}_{:= I}+\frac{1}{2|A|}\nrm{K}_{L^1(B_{2\delta}(0))}\right).
\end{align*}
The integral $I$ is finite independently of $\alpha$, and hence so is the entire expression in parentheses, due to the power-law growth \eqref{powerlaw} and local integrability of $K$ together with the fact that $\critrho$ has finite $p_{_K}$-th moment. Indeed, fixing $R>{R}_{_{K}}+2\delta$, we partition the integral to get  
\begin{align*}
I &=\int_{B_R(x_0)\cap D}\left(\int_A |K(x-y)|\, dy \right)\critrho(x)\,dx+\int_{B_R^c(x_0)\cap D}\left(\int_A |K(x-y)|\,dy \right) \critrho(x)\,dx\\[5pt]
&\leq \nrm{K}_{L^1(B_{R+\delta}(0))} \int_{B_R(x_0)\cap D}\critrho(x)\,dx+C_K|A|\int_{B_R^c(x_0)\cap D}(|x-x_0|+\delta)^{p_{_K}} \critrho(x)\,dx\\[5pt]
&\leq \nrm{K}_{L^1(B_{R+\delta}(0))}+C_K|A|\left(1+\frac{\delta}{R}\right)^{p_{_K}} M^{x_0}_{p_{_K}}(\critrho),
\end{align*}
which is finite. From this bound we have that for some $C>0$ independent of $\alpha$, 
\[\CalK[\eta] < \CalK[\critrho]+\alpha C.\]
For the entropy, since $A \cap \text{supp}(\critrho) = \emptyset$ we have
\begin{align*}
\mathcal{S}[\eta] &= (1-\alpha)\int_D\critrho(x)\log((1-\alpha)\critrho(x))\,dx+\frac{\alpha}{|A|}\int_A\log\left(\frac{\alpha}{|A|}\right)\,dx\\
&=(1-\alpha)\mathcal{S}[\critrho]+(1-\alpha)\log(1-\alpha)+\alpha\log\left(\frac{\alpha}{|A|}\right)\\
&< (1-\alpha)\mathcal{S}[\critrho]+\alpha\log\left(\frac{\alpha}{|A|}\right)\\
&= \mathcal{S}[\critrho] +\alpha\left(-\mathcal{S}[\critrho]+\log\left(\frac{\alpha}{|A|}\right)\right).
\end{align*}
Together this allows us to bound the difference in energy as follows:
\begin{align*}
\mathcal{E}^\nu[\eta]-\mathcal{E}^\nu[\critrho] &< \alpha\left(C -\nu\mathcal{S}[\critrho]+ \nu\log\left(\frac{\alpha}{|A|}\right)-\CalV[\critrho]+\frac{1}{|A|}\int_AV(x)\,dx\right).
\end{align*}
Now, choosing $\alpha$ such that
\[\alpha<|A|\,\exp\left(-\frac{C}{\nu}+\mathcal{S}[\critrho]+\frac{1}{\nu}\CalV[\critrho]-\frac{1}{\nu|A|}\int_AV(x)\,dx\right),\]
along with the constraint \eqref{alphbound1}, we see by the monotonicity of the logarithm that 
\[\mathcal{E}^\nu[\eta]<\mathcal{E}^\nu[\critrho].\]
Since $\eta$ has lower energy than $\critrho$ and lives in the ball $\ballp{p}{\critrho}{r}$, $\critrho$ cannot be a $\CalW_p$-$r$ local minimizer, giving us the desired contradiction. Thus, the support of $\critrho$ must be the entire domain $D$.
\end{proof}


We now derive the Euler-Lagrange equation.
\begin{thm}\label{EL} [Euler-Lagrange equation]
Let Assumptions \ref{assumKV} and \ref{assumD} hold. Suppose that $\critrho\in \CalP^{ac}_p(D)$ is a $\CalW_p$-$r$ local extremizer of $\CalE^\nu$ for some $p\in \big[1, \infty\big]$. Then there exists a constant $\lambda\in \RR$ such that
\begin{equation}\label{ELcond}
K*\critrho(x) +\nu\log(\critrho(x)) +V(x)= \lambda \quad \text{ for } \quad \critrho\text{-a.e. } x\in D.
\end{equation}
\end{thm}
\begin{proof} 
Without loss of generality, assume $\critrho$ is a $\CalW_p$-$r$ local \textit{minimizer} (the case where $\critrho$ is a maximizer follows similarly by reversing the following inequality). As in \cite{carrillo2018existence}, it follows that
\[\frac{d}{dt}\mathcal{E}_\nu[\critrho+t(\eta-\critrho)]\Bigg\vert_{t=0}\geq 0\]
for all $\eta\in \ballp{p}{\critrho}{r}$. From this a direct calculation then yields
\begin{equation} \label{ELbound}
 \int_D \left(K*\critrho+\nu\log(\critrho)+V\right)d\eta\geq \int_D
\left(K*\critrho+\nu\log(\critrho)+V\right)d\critrho.
\end{equation}
We now construct a suitably general $\eta$ to deduce \eqref{ELcond}. Choose $\phi$ in $L^\infty(D;\,\critrho)\cap L^1(D;\,\critrho)$
and define
\[
\eta = \critrho+\epsilon\left(\phi-\int_D\phi \,d\critrho\right)\critrho,
\]
where $\epsilon$ will be chosen such that $\eta \in \ballp{p}{\critrho}{r}$. It is clear that $\eta(D)=1$. To ensure that $\eta\geq 0$ and hence $\eta\in\mathcal{P}_p(D)$, it suffices to pick $\epsilon \leq \frac{1}{2\nrm{\phi}_\infty}$. Another application of Lemma \ref{ELlemm2} gives
\begin{align*}
\CalW_p^p(\critrho,\eta)&\leq 2^{p-1}\int_D|x|^p\,d|\critrho-\eta|\\
&= \epsilon \,2^{p-1}\int_D|x|^p\left\vert\phi - \int_D\phi \,d\critrho\right\vert \,d\critrho\\
&\leq \epsilon\, 2^p\nrm{\phi}_{\infty}M_p(\critrho).
\end{align*}   
Hence, $\CalW_p(\critrho,\eta) <r$ provided 
\[\epsilon <\min\left\{\dfrac{r^p}{2^p \nrm{\phi}_{\infty} M_p(\critrho)},\,\frac{1}{2\nrm{\phi}_\infty}\right\},\] 
which guarantees that $\eta \in \ballp{p}{\critrho}{r}$. Substituting $\eta$ into \eqref{ELbound} 
then gives us
\[\int_D \left(\phi-\int_D\phi \,d\critrho\right)\left(K*\critrho+\nu\log(\critrho)+V\right)\,d\critrho\geq 0.\]
The above calculations work for both $\phi$ and $-\phi$, hence upon multiplying by $-1$ we find that 
\[ \int_D \left(\phi-\int_D\phi\, d\critrho\right)\left(K*\critrho+\nu\log(\critrho)+V\right)\,d\critrho = 0.\]
Now, by setting $\phi = \ind{B}$ for any Borel set $B\subset \supp{\critrho}$ with $\critrho(B) >0$, we further have
\begin{equation}\label{ELcondvanish}
\frac{1}{\critrho(B)}\int_B \left(\nu\log(\critrho)+K*\critrho+V\right)d\critrho
 = \int_D \left(\nu\log(\critrho)+K*\critrho+V\right)d\critrho.
\end{equation}
From this we deduce \eqref{ELcond} by contradiction. Define 
\begin{equation}\label{Lambda}
\Lambda(x) := K*\critrho(x)+\nu\log(\critrho(x))+V(x)
\end{equation}
and assume that $\Lambda$ is not constant $\critrho$-a.e. Then there exists $\lambda^*\in \RR$ such that the sets $B_1 = \left\{\Lambda< \lambda^*\right\}$ and $B_2 = \left\{\Lambda> \lambda^*\right\}$ satisfy $\critrho(B_1)>0$ and $\critrho(B_2)>0$. 
Using $B=B_1$ and $B=B_2$ in \eqref{ELcondvanish} then gives us
\[\lambda^* \,>\,\int_D\Lambda(x)\,d\critrho\txt{0.5}{and}\lambda^* \,<\,\int_D\Lambda(x)\,d\critrho,\]
respectively, which is a contradiction, thus $\Lambda$ must be constant $\critrho$--a.e. This completes the proof. 
\end{proof}


\paragraph{Fixed-Point Characterization.} The Euler-Lagrange equation \eqref{ELcond} can be recast in the following way if the critical point $\gmin$ satisfies $\supp{\gmin} =D$. Solving for $\critrho$ using the logarithm we have 
\[\critrho(x) = \frac{1}{Z(\critrho)}\exp\left(-\frac{K*\critrho(x)+V(x)}{\nu}\right),\]
where
\begin{equation}\label{gibbstatZ}
\hspace{4pt}Z(\critrho) := \int_D \exp\left(-\frac{K*\critrho(x)+V(x)}{\nu}\right)\, dx.
\end{equation}
This motivates the following corollary which will be used below. 
\begin{corr}\label{fixedpointmap}
Let $\critrho \in \CalP^{ac}(D)$ have $\supp{\critrho} = D$. Then $\critrho$ satisfies \eqref{ELcond} if and only if $\critrho$ is a fixed point of the map $T: \CalP(D)\to\CalP^{ac}(D)$ defined by 
\begin{equation}\label{gibbstat}
T(\mu) = \frac{1}{Z(\mu)}\exp\left(-\frac{K*\mu(x)+V(x)}{\nu}\right)
\end{equation}
for $Z$ defined in \eqref{gibbstatZ}.
\end{corr}
By integrating \eqref{ELcond} against $d\critrho(x)$, we can also identify the constant $\lambda$ as  
\[\lambda = \CalE^\nu[\critrho]+\CalK[\critrho] = -\nu\log\left(Z(\critrho)\right).\]
In Section \ref{sect:numerics} we discretize \eqref{gibbstat} for numerical computation of critical points. We note that \eqref{gibbstat} has been used in the literature, for instance by Benachour et al. in \cite{benachour1998nonlinear} to show existence of stationary states for associated McKean-Vlasov processes on $D=\RR$. 
%
%
%
%
%


\section{Review: Existence of Global Minimizers in Free Space}
\label{sect:free-space}
To exhibit the role played by domain geometry in determining existence of global minimizers, we briefly review existence results in free space. In \cite{carrillo2018existence}, the authors show that when $D=\Rd$ and $V=0$, existence of a global minimizer is guaranteed as soon as the energy is bounded below. As we will show, this is not the case in domains with boundaries.  

Unboundedness from below of the energy is shown in \cite{carrillo2018existence} to correspond to an imbalance of diffusive and aggregative forces. If local attractive forces are too strong with respect to local diffusive repulsion, then the energy is lowered to $-\infty$ as the swarm aggregates onto a discrete set of points. It is shown in Theorem 4.1 of \cite{carrillo2018existence} that if 
\begin{equation}\label{NEaggreg}
\liminf_{|x|\to 0} \nabla K(x)\cdot x > 2d\nu,
\end{equation}
then such aggregation-dominated contraction occurs and $\inf \CalE^\nu =-\infty$. Condition \eqref{NEaggreg} is shown to imply that 
\[K \lesssim 2d\nu \log|x| \txt{0.4}{as} |x|\to 0,\]
hence aggregation-dominated contraction may occur unless $K$ is well-behaved at the origin; in particular $K$ cannot have a singularity at the origin worse than logarithmic.   

If diffusion is too strong with respect to long-range attractive forces, then minimizing sequences of $\CalE^\nu$ vanish as diffusion causes infinite spreading of the swarm throughout the domain. For linear diffusion, Carrillo et al. \cite{carrillo2018existence} show that existence of global minimizers of $\CalE^\nu$  in free space for $V=0$ corresponds to the following conditions on $K$, $\nu$ and the dimension $d$, which prevents both diffusion-dominated spreading and aggregation-dominated contraction. Let $K$ satisfy Assumption \ref{assumKV} and be positive and differentiable away from the origin. If $K$ satisfies
\begin{equation}\label{nonexistcond1}
\limsup_{|x|\to \infty}\,\nabla K(x)\cdot x<2d\nu,
\end{equation}
then $\CalE^\nu$ is not bounded below and a global minimizer does not exist. Alternatively, if
\begin{equation}\label{existcond1}
\liminf_{|x|\to \infty}\,\nabla K(x)\cdot x>2d\nu,
\end{equation}
then $\CalE^\nu$ is bounded below and there exists $\gmin \in \CalP(\Rd)$ such that
\[\CalE^\nu(\gmin)=\inf \CalE^\nu>-\infty.\]
By requiring $K$ to be positive, the condition \eqref{NEaggreg} is (sufficiently) prevented and aggregation-dominated contraction cannot occur. In addition, the constraint \eqref{existcond1}, shown in \cite{carrillo2018existence} to imply
\begin{equation}\label{existcond2}
K(x) \gtrsim 2d \nu \log|x| \txt{0.4}{as}|x|\to \infty,
\end{equation}
prevents diffusion-dominated spreading by requiring that $K$ grows at least logarithmically as $|x|\to \infty$. Moreover, by inspection of condition \eqref{nonexistcond1}, we see that condition \eqref{existcond1} is sharp.

In summary, in free space, if the two force-imbalance pathologies of diffusion-dominated spreading or aggregation-dominated contraction are prevented, then the energy is bounded below, which immediately implies existence of a global minimizer. Moreover, condition \eqref{existcond1} for preventing diffusion-dominated spreading is sharp. 

In domains with boundaries, unboundedness from below of the energy due to aggregation-dominated contraction occurs under the same condition as in free space, \eqref{NEaggreg}, while diffusion-dominated spreading occurs under a condition analogous to \eqref{nonexistcond1}, only with dependence on the effective volume dimension $\eff$ instead of the dimension $d$. Moreover, we will see that such conditions for spreading are sharp in the class of domains defined by \eqref{Fdomain}. In addition, we find that boundedness from below of the energy is not enough to grant existence of a global minimizer when the domain is not suitably symmetric. 
\begin{rmrk}
\normalfont
Requirements on the interaction potential $K$ for existence or non-existence of global minimizers of $\CalE^\nu$ are presented in Theorems \ref{NE}, \ref{tightbound}, \ref{existence1} and \ref{existence2} in the form of asymptotic relations similar to \eqref{existcond2} which provide a more explicit characterization of $K$ than \eqref{existcond1}; however, we could have equivalently worked with conditions such as \eqref{existcond1} involving $\nabla K$. 
\end{rmrk}
\noindent The following lemmas from \cite{carrillo2018existence} will be used in the existence proofs in our paper. 
\begin{lemm}\label{lowerSC}
\cite{carrillo2018existence}
Assume that $K$ and $V$ are both lower semicontinuous. Then $\CalE^\nu$ is weakly-* lower semicontinuous, in that for any sequence $\left\{\mu_n\right\}_{n\geq 0} \subset \CalP(D)$ such that $\mu_n\overset{*}{\rightharpoonup} \mu\in \CalP(D)$, it holds that
\[\liminf_{n\to \infty} \CalE^\nu[\mu_n] \geq \CalE^\nu[\mu].\] 
\end{lemm}
\begin{lemm}{(Logarithmic Hardy-Littlewood-Sobolev (HLS) inequality \cite[Lemma 2.6]{carrillo2018existence})}
\label{lemma:HLS}
Let $\rho\in \CalP^{ac}(\Rd)$ satisfy $\log(1+|\cdot|^2)\rho\in L^1(\Rd)$. Then there exists $C_0\in \RR$ depending only on $d$ such that
\begin{equation}\label{HLS}
-\int_{\Rd}\int_{\Rd}\log(|x-y|)\rho(x)\rho(y)\,dx\,dy\leq \frac{1}{d}\int_{\Rd}\rho(x)\log(\rho(x))\,dx+C_0.
\end{equation}
\end{lemm}
\begin{lemm}{\cite[Lemma 2.9]{carrillo2018existence}}\label{Ktight}
Let $K(x) \in L^1_{loc}(\Rd)$ be positive, symmetric and satisfying
\[\lim_{|x|\to \infty} K(x) = +\infty.\]
Given a sequence $\left\{\mu_n\right\}_{n\geq 0} \subset \CalP(D)$, if
\[\liminf_{n\to\infty}\int_{\Rd}\int_{\Rd}K(x-y)\,d\mu_n(x)\,d\mu_n(y)<\infty,\]
then $\left\{\mu_n\right\}_{n\geq 0} \subset \CalP(D)$ is weakly-* relatively compact up to translations.
\end{lemm}
%
%


\section{Non-existence of Global Minimizers: Domains with Boundaries}
\label{sect:non-exist}
In this section we investigate various possible scenarios when global minimizers of the energy cannot exist. First we treat the force-imbalance pathologies from free space which could make the energy unbounded from below. We then introduce a new non-existence phenomenon which results entirely from asymmetries in the domain and only occurs in domains with boundaries.


\subsection{Imbalance of Forces}
\label{subsect:imbalance}
In comparison with the spreading case in free space \eqref{nonexistcond1}, the following result gives a relation between the diffusion parameter $\nu$, the interaction potential $K$ and the effective volume dimension $\eff$ (as opposed to the dimension of the space $d$) which guarantees non-existence of ground states of $\CalE^\nu$ in the form of diffusion-dominated spreading throughout the domain.
\begin{thm}\label{NE} [Non-existence: diffusion-dominated regime]
Let Assumptions 1 and 2 hold with $V=0$ and $\eff$ defined by \eqref{effvoldim}. Then the energy $\CalE^\nu$ is not bounded below on $\CalP(D)$ provided there exists $\delta_0$ with 
$0<\delta_0<1$, $C_0\in \RR$ and $R_0>0$ such that 
\begin{equation}\label{lemm1i}
K(x) \leq 2(1-\delta_0)\eff\nu\log|x| +C_0, \qquad \text{for all} \quad |x|>R_0.
\end{equation}
\end{thm}
\begin{proof}
We will explicitly construct a sequence of measures which sends the energy to $-\infty$ by exploiting properties of the effective volume dimension. 
The supremum in the function $V_{_D}(r)$ defined in \eqref{fDr}
implies that for every $n\in \NN $ there exists $x_n\in D$ such that 
\[V_{_D}(n)>\left\vert D\cap B_n(x_n)\right\vert> \frac{1}{2} V_{_D}(n).\]
Define the sets $D_n := D\cap B_n(x_n)$ and the sequence of probability measures 
\[\mu_n = \frac{1}{|D_n|}\ind{D_n}.\]
We first bound above the interaction energy of $\mu_n$ for $n\geq R_0$: 
\begin{align*}
\CalK[\mu_n] &= \frac{1}{2|D_n|^2}\int_{D_n}\int_{D_n} K(x-y)\,dy\,dx \\
&=\frac{1}{2|D_n|^2}\int_{D_n}\left[\int_{D_n\cap B_{R_0}(x)}K(x-y)\,dy+ \int_{D_n\cap B^c_{R_0}(x)}K(x-y)\,dy\right]\,dx\\
&\leq \frac{1}{2|D_n|} \bigg[\nrm{K}_{L^1(B_{R_0}(0))}+\left\vert D_n\cap B^c_{R_0}\right\vert\Big(2(1-\delta_0)\eff\nu\log(2n) +C_0\Big)\bigg]\\[5pt]
&\leq (1-\delta_0)\eff\nu\log(n) +\tilde{C}_1
\end{align*}
where 
\[\tilde{C}_1 = \frac{1}{2|D_1|} \nrm{K}_{L^1(B_{R_0}(0))}+(1-\delta_0)\eff\nu\log(2) +\frac{1}{2}C_0.\]
Using the characterization \eqref{effvolbound} of the effective volume dimension, for $n>r'$ we have
\[|D_n| >\frac{1}{2} V_{_D}(n)\geq \frac{C}{2}n^{\eff},\]
and so the entropy of $\mu_n$ for $n>r'$ is bounded above as follows:
\[\CalS[\mu_n] = -\log|D_n| \leq -\log\left(\frac{C}{2}n^{\eff}\right)=-\eff\log\left(n\right)-\log\left(\frac{C}{2}\right).\]
%
%
%
Hence, for $n>\max\left\{R_0,r'\right\}$, the total energy of $\mu_n$ satisfies
\[\CalE^\nu[\mu_n]\ \ \leq\ \ (1-\delta_0)\eff\nu\log(n) - \eff\nu\log\left(n\right) +\tilde{C}_1-\nu\log\left(\frac{C}{2}\right)\ \ =\ \ -\delta_0 \eff\nu\log(n) + \tilde{C},\]\\[-10pt]
for $\tilde{C}\in \RR$. Hence, $\lim_{n\to \infty} \CalE^\nu[\mu_n] = -\infty$, which concludes the proof.
\end{proof}

The interpretation of Theorem \ref{NE} is that if attraction forces are too weak (at large distances), then diffusion dominates and spreading occurs. In free space ($D=\Rd$), the result is consistent with that derived by Carrillo et al. in \cite{carrillo2018existence} since $\eff=d$. In more general unbounded domains the relevant factor in the logarithmic bound \eqref{lemm1i} includes the effective volume dimension, more specifically the product $\eff \nu$. With the interpretation that the effective dimension specifies the number of orthogonal directions that independently extend to infinity in  $D$, the factor $\eff \nu$ suggests a minimal balance between the diffusion $\nu$ and the number of directions $\eff$ in which mass can escape to infinity, so that a global minimizer can exist.
\begin{rmrk}
\normalfont
As mentioned in Section \ref{sect:free-space}, non-existence due to aggregation-dominated contraction occurs in domains with boundaries under the same conditions as in free space (i.e. condition \eqref{NEaggreg}), since contraction to a point can occur in any domain $D\in \CalB^d$ satisfying $|D|>0$.
\end{rmrk}
%

\subsection{Escaping Mass Phenomena}

\label{subsect:escape}
As mentioned above, boundedness from below of the energy is not sufficient to guarantee existence of a minimizer in domains with boundaries. To begin this discussion, we present the following theorem, where \eqref{existcond1} is clearly satisfied, hence $\inf\CalE^\nu>-\infty$, yet no energy minimizer exists.
\begin{thm}\label{COM1}
Let $D = [0, +\infty)$, $K(x) = \frac{1}{2}x^2$ and $V = 0$. Then the energy $\CalE^\nu$ has no minimizers.
\end{thm}
\begin{proof}
\noindent The energy is given by
\begin{equation}\label{en}
\CalE^\nu[\rho] = \frac{1}{4}\int_0^\infty\int_0^\infty (x-y)^2\,\rho(x)\rho(y)\,dx\,dy+\nu\int_0^\infty\rho(x)\log(\rho(x))\,dx.
\end{equation}
We proceed by contradiction. Assume that a minimizer $\gmin$ of \eqref{en} exists. Then $\gmin\in \CalP_2^{ac}(D)$ due to the growth of $K$. Since $\gmin$ has $\supp{\gmin} = D$ by Theorem \ref{suppmin} and $\gmin$ satisfies the Euler-Lagrange equation \eqref{ELcond}, by Corollary \ref{fixedpointmap} one has
\[\gmin(x) = Z^{-1}\exp\left(-\frac{K*\gmin(x)}{\nu}\right),\]
where 
\[Z = \int_0^\infty \exp\left(-\frac{K*\gmin(x)}{\nu}\right)\,dx.\] 
From an elementary calculation,  
\begin{align*}
K*\gmin(x) &= \frac{1}{2}\int_0^\infty(x-y)^2\gmin(y)\,dy\\ 
&= \frac{1}{2}(x-M_1(\gmin))^2-\frac{1}{2}\left(M_1(\gmin)^2-M_2(\gmin)\right),
\end{align*}
hence $\gmin = \rho_c$ for some $c\in \RR$, where
\[\rho_c(x) = A(c)\,\exp\left(-\frac{1}{2\nu}\left(x-c\right)^2\right)\]
is a shifted and truncated Gaussian. Here $c = M_1(\gmin)$ and $A(c)$ is the normalization constant
\[A(c) = \frac{2/\sqrt{2\pi\nu}}{1+\text{erf}(c/\sqrt{2\nu})},\]
where $\text{erf}(x)$ denotes the error function.

Let $\Gamma_c = \left\{\rho_c\right\}_{c\geq 0}$ be the family of shifted and truncated Gaussians on $[0,+\infty)$. Then since $\gmin\in\Gamma_c$ and $\gmin$ is a critical point of $\CalE^\nu$ over $\CalP(D)$, $\gmin$ is a critical point of $\CalE^\nu$ over $\Gamma_c$ as well, and so the function $c\to\CalE^\nu[\rho_c]$ has a critical point at some $c \in \RR$. By direct calculation of $\CalE^\nu[\rho_c]$, we now show that no such critical point exists.

For the entropy, we have
\begin{align*}
\CalS[\rho_c] &= A(c) \int_0^\infty \exp\left(-\frac{1}{2\nu}\left(x-c\right)^2\right)\left(-\frac{1}{2\nu}\left(x-c\right)^2+\log(A(c))\right)\,dx\\
&=\log(A(c))-\frac{1}{2\nu}\underbrace{A(c)\int_0^\infty (x-c)^2\exp\left(-\frac{1}{2\nu}(x-c)^2\right)\,dx}_{I}.\\
\intertext{For the interaction energy, we get}
\CalK[\rho_c] &= \frac{1}{4}A(c)^2\int_0^\infty\int_0^\infty(x-y)^2 \exp\left\{-\frac{1}{2\nu}(x-c)^2-\frac{1}{2\nu}(y-c)^2\right\} \,dx\,dy\\
&= \frac{1}{2}\underbrace{A(c)\int_0^\infty (x-c)^2\exp\left(-\frac{1}{2\nu}(x-c)^2\right)\,dx}_{I}\\ 
&\mathbin{\color{white}===}-\frac{1}{2}\left[A(c)\int_0^\infty (x-c)\exp\left(-\frac{1}{2\nu}(x-c)^2\right)\,dx\right]^2\\
&= \frac{1}{2}I - \frac{\nu^2}{2}\left(A(c) \exp\left(-\frac{c^2}{2\nu}\right)\right)^2.
\end{align*}
The total energy $\CalE^\nu[\rho_c]$ then reduces to
\begin{align*}
\CalE^\nu[\rho_c] &= \CalK[\rho_c]+\nu\CalS[\rho_c]\\
&= \nu\log(A(c)) -\frac{\nu^2}{2}\left(A(c) \exp\left(-\dfrac{c^2}{2\nu}\right)\right)^2\\
&= \nu\log\left(\frac{2}{\sqrt{2\pi\nu}}\right)-\nu\log\left(1+\text{erf}\left(\frac{c}{\sqrt{2\nu}}\right)\right)
-\frac{\nu}{4}\left(\dfrac{\dfrac{2}{\sqrt{\pi}}\exp\left(-\dfrac{c^2}{2\nu}\right)}{1+\text{erf}\left(\dfrac{c}{\sqrt{2\nu}}\right)}\right)^2.
\end{align*}
By letting $\tilde{c} = c/{\sqrt{2\nu}}$, one can also write:
\begin{equation}\label{COMenergy1}
\CalE^\nu[\rho_c]=\nu\log\left(\frac{2}{\sqrt{2\pi\nu}}\right) - \nu\left(f(\tilde{c})+\frac{1}{4}f'(\tilde{c})^2\right),
\end{equation}
where 
\[f(\tilde{c}) = \log\left(1+\text{erf}(\tilde{c})\right).\]
Since $c\to \CalE^\nu[\rho_c]$ has a critical point (by hypothesis) and is a smooth function, we have
\begin{equation}\label{COMcont}
\frac{d}{dc}\CalE^\nu[\rho_c] = -\sqrt{\frac{\nu}{2}} f'(\tilde{c})\left(1+\frac{1}{2}f''(\tilde{c})\right) = 0, \\[5pt]
\end{equation}
for some $\tilde{c}\in\RR$. However, $f' > 0$ and $\min f'' = -\frac{4}{\pi}>-2$ together imply that \eqref{COMcont} has no solutions. This contradicts the assumption that $\CalE^\nu$ has a critical point. 

Figure \ref{energy_quad_attract} shows the monotonically decreasing profile of $c\to \CalE^\nu[\rho_c]$ together with energy plots with an added external potential $V(x)=gx$, which will be addressed in Theorem \ref{COM2}.  In particular, it shows the case $g=0$ corresponding to \eqref{COMenergy1}.
\end{proof}

\begin{figure}[htp]
\centering
\includegraphics[trim={10 5 40 20},clip,width=0.8\textwidth]{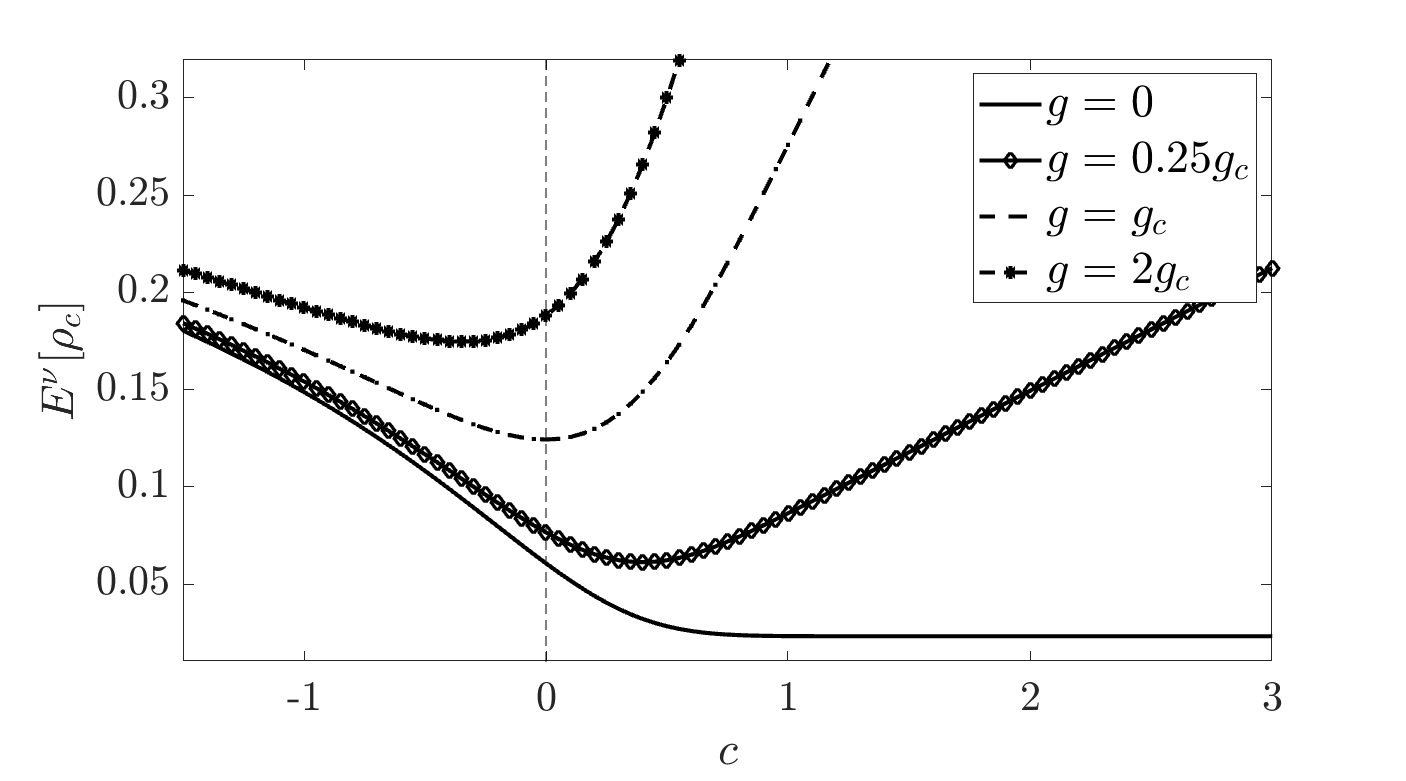} \\
\caption{Plots of the energy $\CalE^\nu$ given by equation \eqref{energy_quad_attract_formula}, evaluated on the set of truncated Gaussians $\Gamma_c$, for several values of $g$. For zero gravity ($g=0$) -- see also \eqref{COMenergy1} -- we see that the energy is monotonically decreasing, and hence has no local minimum, while for $g>0$ there exists a unique minimum value of $c$. Moreover, for $g=g_c:=\sqrt{\frac{2\nu}{\pi}}$ we see that $c=0$, as in Remark \ref{rmk:gc}, which implies that the minimizer $\rho_c$ is a half-Gaussian.}
\label{energy_quad_attract}
\end{figure}

The non-existence result in Theorem \ref{COM1} is an example of a more general phenomenon which we refer to as the \textit{escaping mass phenomenon}. Sequences such as $\Gamma_c$ in Theorem \ref{COM1} are \textit{escaping} in the sense that the centre of mass $\com\left(\rho_c\right)$ reaches infinity without the measures vanishing in the traditional sense. This phenomenon manifests in dynamics as the persistent, metastable translation of the centre of mass of the swarm (see Remark \ref{COMdyn}). 

Geometrically, asymmetries in $D$ cause the energy $\CalE^\nu$ to lose translation invariance, and so the same tight-up-to-translations arguments from free space do not apply. To enforce the existence of a global minimizer, we can add a confining potential $V$ and exploit any symmetries within $D$. This process is described in Section \ref{sect:existence}. To complete the discussion on escaping mass, Theorem \ref{domassymthm} below provides a necessary condition for the existence of a minimizer which comes as a direct corollary of the Euler-Lagrange equation.

%
\begin{thm}\label{domassymthm}
Let Assumptions 1 and 2 be satisfied and let $K,V\in W_{loc}^{1,1}(D)$. Then if $\gmin$ is a critical point of the energy $\,\CalE^\nu$ with $\supp{\gmin} = D$, then $\gmin$ satisfies 
\begin{equation}\label{boundaryV}
\nu\int_{\partial D} n(x)\,\gmin(x)\,dS(x) = -\int_D  \gmin(x)\nabla V(x)\,dx,
\end{equation}
%
\end{thm}
\begin{proof}
Assume that $\gmin$ is such a critical point. 
%
%
%
%
Since $K,V\in W_{loc}^{1,1}(D)$, $\gmin$ is differentiable almost everywhere. Taking the gradient of both sides of the Euler-Lagrange equation \eqref{ELcond} and integrating against $\gmin(x)\,dx$ then gives us 
\begin{equation}\label{domassymthm_grad}
\nu\int_D \nabla \gmin(x) \,dx = -\underbrace{\int_D \gmin(x) \nabla K*\gmin(x)\, dx}_{=:I}- \int_D \gmin(x)\nabla V(x)\,dx.
\end{equation}
The anti-symmetry of $\nabla K$ implies that
\[I = \int_D\int_D \nabla K(x-y)\,\gmin(y)\,\gmin(x)\,dy\,dx = -\int_D\int_D \nabla K(y-x)\,\gmin(y)\,\gmin(x)\,dy\,dx =-I,\]
%
and hence $I=0$. To integrate the left-hand side of \eqref{domassymthm_grad}, we use the divergence theorem. Consider a sequence of bounded sets $A_n \subset D$ with smooth boundary such that $\lim_{n\to\infty} A_n = D$ and $\lim_{n\to\infty} \partial A_n = \partial D$. Then for any fixed $\vec{a}\in \Rd$,
\begin{align*}
\vec{a}\cdot \int_{A_n}\nabla \gmin(x)\,dx &= \int_{A_n}\nabla\cdot\left(\gmin(x)\vec{a}\right)\,dx\\ 
&= \int_{\partial A_n}n(x)\cdot \left(\gmin(x)\vec{a}\right)\,dS(x)\\ 
&= \vec{a}\cdot\int_{\partial A_n}n(x)\,\gmin(x)\,dS(x).
\end{align*}
Since this holds for any $\vec{a}\in \Rd$, for each $n$ we have 
\[\int_{A_n} \nabla \gmin(x)\,dx = \int_{\partial A_n} n(x)\gmin(x)\,dS(x),\]
and so
%
%
%
\[\int_D \nabla \gmin(x)\,dx = \lim_{n\to \infty} \int_{A_n} \nabla \gmin(x)\,dx = \lim_{n\to \infty} \int_{\partial A_n} n(x)\gmin(x)\,dS(x) = \int_{\partial D} n(x)\, \gmin(x)\,dS(x).\]
Since $n(x)$ is defined for almost every $x\in \partial D$ (Assumption \ref{assumD}) and $\gmin\in L^1(D)$, we can apply classical trace theorems \cite[Ch. 5]{evans10} to conclude that the right-most integral in \eqref{domassymthm_grad} is finite. This yields the result.
\end{proof}
\begin{rmrk}\label{nonexist}
\normalfont
Theorem \ref{domassymthm} indicates that minimizers of $\CalE^\nu$ cannot exist under zero external potential in a large class of domains (see the Examples below). Indeed, $V=0$ implies that the right-hand side of \eqref{boundaryV} is zero, yet the left-hand side of \eqref{boundaryV} is nonzero: the formula for critical points $\gmin$ with $\supp{\gmin}=D$,  
\[\gmin(x) = Z^{-1}\exp\left(-\frac{K*\gmin(x)}{\nu}\right),\]
implies that $\gmin(x)>0$ for all $x\in \partial D$. For this reason, \eqref{boundaryV} cannot hold in many infinite domains. 
\end{rmrk}
\vspace{10pt}
\begin{rmrk}\label{COMdyn}
\normalfont Condition \eqref{boundaryV} relates to the dynamics of the aggregation-diffusion model \eqref{aggdD} in the following way. Consider the evolution in time of the centre of mass:
\begin{align*}
\frac{d}{dt}\mathcal{C}(\munu) &=\int_Dx\left(\ppt\rhonu(x)\right)\, dx\\
&= \int_Dx \nabla \cdot\bigg(\rhonu(x)\Big(\nabla K*\rhonu(x)+\nabla V(x)\Big)+\nu\nabla \rhonu(x)\bigg)\,dx\\
&= -\int_D\nabla K*\rhonu(x) \rhonu(x)\,dx-\int_D\nabla V(x)\rhonu(x)\,dx-\nu\int_{D} \nabla \rho^{\nu}_t(x)\,dx\\
\intertext{(integrating by parts and utilizing the boundary conditions)}
&=-\int_D \nabla V(x)\rhonu(x)\,dx-\nu\int_{\partial D} n(x) \rho^{\nu}_t(x) \,dS(x).
\end{align*}
For $V=0$, this is exactly
\begin{equation}\label{comtrans}
\frac{d}{dt}\com\left(\munu\right) = -\nu\int_{\partial D} n(x) \rho^{\nu}_t(x) \,dS(x),
\end{equation}
hence the swarm translates in the direction opposite the average outward normal vector with speed proportional to the mass along the boundary, weighted by $\nu$. Unless the domain is bounded or symmetric enough that mass may be distributed along the boundary in such a way that the right-hand side of \eqref{comtrans} is zero, translation will occur indefinitely, further justifying the terminology ``escaping-mass phenomenon''. Clearly this takes effect as soon as $\nu>0$.
\end{rmrk}
\subsubsection{Examples}

\hspace{12pt} The following are a few example domains where a minimizer $\gmin$ cannot exist by the argument in Remark \ref{nonexist}.
\begin{enumerate}
\item Half-space: Here $D = \Rd_+:= \RR^{d-1}\times [0,\infty)$ where $n(x) = -\hat{e}_d$ is constant for all $x\in \partial D$. This gives
\[\int_{\partial D}n(x)\,\gmin(x)\,dS(x) \,=\, -\hat{e}_d\int_{\RR^{d-1}}\gmin(x)\,dx_1\dots \,dx_{d-1} \,<\, 0.\]
Note that Theorem \ref{COM1} demonstrates this case for $d=1$.
\item Wedge domain: $D = \left\{x\in \RR^2 \txt{0.2}{:} 0\leq x_2\leq \tan(\phi) x_1\right\}$ for $\phi\in (0,\pi/2)$. Then 
\[\int_{\partial D}n(x)\,\gmin(x)\,dS(x) \,=\, (N_1,N_2)\]
where 
\[N_1 = -\sin(\phi)\int_0^\infty \gmin(z,\tan(\phi)z)\,dz<0.\]
\item Paraboloid: Let $x = (x_1,\dots,x_{d-1},x_d) = (x',x_d)\in \Rd$ and define  
\[D = \left\{x\in \Rd \txt{0.2}{:} x_d \geq |x'|^2\right\}.\]
Then $n(x) = \dfrac{1}{\sqrt{|x'|^2+\frac{1}{4}}}\left(x', -\frac{1}{2}\right)$ and so 
\[\int_{\partial D} n(x)\,\gmin(x) \, dS(x) = (N', N_d)\]
where again $N_d$ cannot be zero.
\end{enumerate}
In the next section we establish a relation between the 
domain geometry and the external potential $V$, motivated by Theorem \ref{domassymthm}, that ensures existence of a minimizer. 


\section{Existence of Global Minimizers}
\label{sect:existence}


\subsection{A sharp existence condition for certain domains} 
\label{subsect:tight-exist}

Recall that in free space we have a sharp condition for existence of global minimizers. More specifically, in free space the sharp condition \eqref{existcond1} (resp. \eqref{nonexistcond1})  determines whether the energy $\CalE^\nu$ is bounded (resp. unbounded) below, and boundedness from below is all that is needed to guarantee existence of minimizers in free space. Here we show that for a wide class of domains with boundaries (i.e. those of the form \eqref{Fdomain}), the analogous condition \eqref{kbound} (resp. \eqref{lemm1i}) for granting boundedness (resp. unboundedness) from below of the energy is also sharp, and depends explicitly on the effective volume dimension $\eff$ of the domain. 
\begin{thm}\label{tightbound}
Let $D\subset \Rd$ have the form \eqref{Fdomain} and $\nu>0$. Suppose that $V=0$ and $K$ is positive and satisfies Assumption \ref{assumKV}. In addition, suppose that for some $\delta_1>0$ and $C_1\in \RR$,
\begin{equation}\label{kbound}
K(x) \geq 2(1+\delta_1)\eff\nu\log|x|+C_1\txt{0.4}{for all} x\in D-D,
\end{equation}
where $D-D := \left\{x-y \in \Rd\txt{0.2}{:} x,y\in D\right\}$. Then the energy $\CalE^\nu$ is bounded below on $\CalP(D)$. Moreover,  there exists a global minimizer $\gmin \in \CalP^{ac}(D)$ of $\CalE^\nu$, that is
\[\CalE^\nu[\gmin] = \inf_{\rho\in \CalP(D)} \CalE^\nu[\rho] >-\infty.\]
\end{thm}
%
%
\begin{proof}
First we establish some notation. Recall that $D=F\times \RR^{d-m}$ where $F\subset \RR^m$ is compact and $m$-dimensional for $m\in \left\{1,\dots,d-1\right\}$. Denote $x = (x_1,x_2,\dots,x_d) \in D$ by $x=(\overline{x},\tilde{x})$ for $\overline{x}\in F$ and $\tilde{x}\in \RR^{d-m}$. For $\rho\in \CalP^{ac}(D)$, define the $\overline{x}$-marginal $\rho_{_{F}} \in \CalP^{ac}(\RR^{d-m})$ of $\rho$ by  
\begin{equation}
\rho_{_{F}}(\tilde{x}) = \int_F \rho(x)\,d\overline{x}.
%
%
\end{equation}
%
%
\textit{Step 1:} For any $\rho\in \CalP^{ac}(D)$, we have 
\begin{equation}\label{step1}
\CalS[\rho]\geq \CalS[\rho_{_{F}}]-\log|F|,
\end{equation}
where 
\[\CalS[\rho_{_{F}}] = \int_{\RR^{d-m}}\rho_{_{F}}(\tilde{x})\log(\rho_{_{F}}(\tilde{x}))\,d\tilde{x}.\]
To show \eqref{step1}, by Fubini's theorem we have
\[\CalS[\rho] =\int_D\rho(x)\log(\rho(x))\,dx= \int_{\R^{d-m}}\left(\int_F\rho(\overline{x},\tilde{x})\log\left(\rho(\overline{x},\tilde{x})\right)\,d\overline{x}\right)\,d\tilde{x}.\]
We now claim that for almost every $\tilde{x}\in \RR^{d-m}$,
\begin{equation}\label{claim1}
\int_F\rho(\overline{x},\tilde{x})\log\left(\rho(\overline{x},\tilde{x})\right)\,d\overline{x} \ \geq\  
\rho_{_{F}}(\tilde{x})\log\left(\frac{\rho_{_{F}}(\tilde{x})}{|F|}\right).
\end{equation}
Assuming the claim, we then have
\[\CalS[\rho] \geq \int_{\RR^{d-m}} \rho_{_{F}}(\tilde{x})\log\left(\frac{\rho_{_{F}}(\tilde{x})}{|F|}\right) \,d\tilde{x} = \CalS[\rho_{_{F}}] - \log|F|,\]
showing \eqref{step1}.

We now prove claim \eqref{claim1} using convexity. For almost every $\tilde{x}\in \RR^{d-m}$, the function $f({\overline{x}}):=\rho({\overline{x}},\tilde{x})$ is defined for $\overline{x}\in F$ up to a set of measure zero and satisfies $\nrm{f}_{L^1(F)} = \rho_{_{F}}(\tilde{x})$. For $\nrm{f}_{L^1(F)} = 0$ or $\nrm{f}_{L^1(F)} = +\infty$, the claim \eqref{claim1} trivially holds with equality. For $0<\nrm{f}_{L^1(F)}<+\infty$, by the convexity of $U(x) = x\log(x)$ we have for almost every $x \in [0,\,\infty)$ and every $y\in [0,\infty)$:
\[U(y) \geq U(x)+U'(x)(y-x).\]
Letting $y = f({\overline{x}})$ and $x = \frac{\nrm{f}_{L^1(F)}}{|F|}$, we have for almost every $\overline{x}\in F$,
\[U\left(f({\overline{x}})\right)\geq U\left(\frac{\nrm{f}_{L^1(F)}}{|F|}\right)+U'\left(\frac{\nrm{f}_{L^1(F)}}{|F|}\right)\,\left(f(\overline{x})-\frac{\nrm{f}_{L^1(F)}}{|F|}\right).\]
Integrating over $F$ we then get
\begin{align*}
\int_FU\left(f({\overline{x}})\right)\,d\overline{x}&\geq \int_FU\left(\frac{\nrm{f}_{L^1(F)}}{|F|}\right)\,d\overline{x}\\ 
&= \int_F \left(\frac{\nrm{f}_{L^1(F)}}{|F|}\right)\log\left(\frac{\nrm{f}_{L^1(F)}}{|F|}\right)\,d\overline{x}\\
&= \rho_{_{F}}(\tilde{x})\log\left(\frac{\rho_{_{F}}(\tilde{x})}{|F|}\right),
\end{align*}
which proves the claim.\\[5pt] 
We briefly note that intuition for \eqref{claim1} comes from the case $\rho_{_{F}}(\tilde{x}) = 1$, which implies that $\rho(\overline{x},\tilde{x})$ is a probability density on $F$. The inequality \eqref{claim1} is then equivalent to the uniform distribution on $F$ being the global minimizer of $\CalS$ over $\CalP(F)$, which is intuitively clear from an information perspective: the uniform distribution corresponds to the state with \textit{least information}, or maximum entropy $-\CalS$.  

%
\noindent \textit{Step 2.} There exists $\tilde{C}\in \RR$ such that for any $\rho\in \CalP^{ac}(D)$, we have
\begin{equation}\label{step2}
\CalK[\rho] \geq -(1+\delta_1)\nu\,\CalS[\rho_{_{F}}] + \tilde{C}.
\end{equation}
To show this, first note that for any $x = (\overline{x},\tilde{x})$ and $y = (\overline{y},\tilde{y})$ in $D$: 
\[\log\left(|x-y|\right)\geq \log\left(|\tilde{x}-\tilde{y}|\right).\]
Using this, together with the lower bound \eqref{kbound} on $K$, Fubini's theorem, and the logarithmic-HLS inequality on the space $\CalP^{ac}(\RR^{d-m})$ (see Lemma \ref{lemma:HLS}), we get:
\begin{align*}
\CalK[\rho] &= \frac{1}{2}\int_D\int_D K(x-y)\rho(x)\rho(y)\,dx\,dy\\ 
&\geq (1+\delta_1)(d-m)\nu\int_D\int_D\log\left(|x-y|\right)\rho(x)\rho(y)\,dx\,dy+ \frac{C_1}{2}\\
&\geq (1+\delta_1)(d-m)\nu\int_{\RR^{d-m}}\int_{\RR^{d-m}}\log\left(|\tilde{x}-\tilde{y}|\right)\left(\int_F\rho(x)\,d\overline{x}\right)\left(\int_F\rho(y)\,d\overline{y}\right)\,d\tilde{x}\,d\tilde{y}+  \frac{C_1}{2}\\
&= (1+\delta_1)(d-m)\nu\int_{\RR^{d-m}}\int_{\RR^{d-m}}\log\left(|\tilde{x}-\tilde{y}|\right)\rho_{_{F}}(\tilde{x})\rho_{_{F}}(\tilde{y})\,d\tilde{x}\,d\tilde{y}+  \frac{C_1}{2}\\
&\geq -(1+\delta_1)\nu\CalS[\rho_{_{F}}] -(1+\delta_1)(d-m)\nu\,C_0+ \frac{C_1}{2}. \\[-10pt]
\end{align*}
%
%
\textit{Step 3.} We now show that $\inf\CalE^\nu>-\infty$. \\
Consider a minimizing sequence $\left\{\rho^n\right\}_{n\geq 0}$ of $\CalE^\nu$ and without loss of generality assume that $\sup_n\left\{\CalE^\nu[\rho^n]\right\}<+\infty$. By Step 1 and the positivity of $K$, for every $n$ we have
\[\nu \CalS[\rho_{_{F}}^n]\leq \nu\CalS[\rho^n]+\nu\log|F|\leq \CalE^\nu[\rho^n]+\nu\log|F|<+\infty\]
and so $\sup_n\left\{\CalS[\rho_{_{F}}^n]\right\}<+\infty$. Putting Steps 1 and 2 together (for $\tilde{C}$ different from above), we get
\[\CalE^\nu[\rho^n] \geq -(1+\delta_1)\nu\CalS[\rho_{_{F}}^n] +\nu \CalS[\rho_{_{F}}^n]+\tilde{C} = - \left(\delta_1\nu\right) \CalS[\rho_{_{F}}^n] +\tilde{C},\]
which implies 
\[\inf_{\rho\in\CalP(D)}\CalE^\nu[\rho] =\lim_{n\to \infty} \CalE^\nu[\rho^n] \geq -  \left(\delta_1\nu\right) \sup_n\left\{\CalS[\rho_{_{F}}^n]\right\} >-\infty.\]
\textit{Step 4.} Minimizing sequences are tight-up-to-translations in $\RR^{d-m}$.\\[5pt]
The inequality \eqref{step2} in Step 2 can be rewritten as
\[\nu\CalS[\rho_{_{F}}]\geq -\frac{1}{1+\delta_1}\left(\CalK[\rho]-\tilde{C}\right).\] 
Combined with Step 1 and the positivity of $K$, this gives us 
\[\CalE^\nu[\rho] \geq \frac{\delta_1}{1+\delta_1}\CalK[\rho]+\tilde{C}\geq \tilde{C},\]
for $\tilde{C}\in \RR$ different from above. The boundedness of $\left\{\CalE^\nu[\rho^n]\right\}_{n\geq 0}$ implies that $\left\{\CalK[\rho^n]\right\}_{n\geq 0}$ is bounded: it was shown above that energy is bounded below, while the interaction energy has to be bounded above, and so by Lemma \ref{Ktight}, there exists a sequence $\left\{\rho^n\right\}_{n\geq 0}$ which is tight up to translations in $\Rd$. Let $\left\{\tilde{\rho}^n\right\}_{n\geq 0}$ be a translated version of the sequence that is tight in $\CalP(\Rd)$ and given by 
\[\tilde{\rho}^n(x) = \rho^n(x-x^n).\] 
Without loss of generality, we may assume by the compactness of $F$ that the translations $x^n$ satisfy $x^n_i = 0$ for $i=1,\dots,m$, which implies that (i) $\left\{\tilde{\rho}^n\right\}_{n\geq 0}$ is tight in $\CalP(D)$ and (ii) for each $n$, $\CalE^\nu[\rho^n] = \CalE^\nu[\tilde{\rho}^n]$, since the energy is invariant to translations in the last $d-m$ coordinates. By (i) and Prokhorov's theorem, we are guaranteed existence of a subsequence $\left\{\tilde{\rho}^{n_k}\right\}_{k\geq 0}$ which converges weakly-* to some $\gmin\in \CalP(D)$, and by (ii) and the lower semicontinuity of the energy (Lemma \ref{lowerSC}), we have that
\[\CalE^\nu[\gmin]\leq \liminf_{n\to \infty}\CalE^\nu[\tilde{\rho}^n] = \lim_{n\to \infty} \CalE^\nu[\rho^n] = \inf_{\rho\in \CalP(D)}\CalE^\nu[\rho],\]
and so $\gmin$ realizes the infimum of $\CalE^\nu$. Since $\inf_{\rho\in \CalP(D)}\CalE^\nu<+\infty$, we have that $\gmin$ is absolutely continuous with respect to the Lebesgue measure (see Remark \ref{ac-gmin}).
\end{proof}

\begin{rmrk}
\normalfont
Boundedness from below of the energy in Theorem \ref{tightbound} can be extended to domains $D = F\times H$ where $F\subset \RR^{m}$ is compact and $m$-dimensional and $H= H_1\times \dots\times H_{d-m}$ where each $H_i\subset \RR$ satisfies $|H_i| = +\infty$ and is given by the closure of a disjoint union of intervals:
\[H_i = \overline{\cup_{k=0}^\infty I_k^i}.\]   
In particular, we could have $H_i = [0,\infty)$. Introducing domains with such asymmetries, however, leads us again into the dilemma of the escaping mass phenomenon, and so boundedness from below of the energy may not be enough to grant existence.
\end{rmrk}


\subsection{Existence of Minimizers via Confining Potentials}
\label{subsect:V-existence}
In this section we establish sufficient conditions for existence of  global minimizers of the energy that take into account the escaping mass phenomenon. Given the considerations above, in many canonical domains we have no global minimizer despite $\CalE^\nu$ being bounded below. With insight from Theorem \ref{domassymthm}, we present here an approach for guaranteeing existence of a global minimizer through the addition of a suitable external potential $V$. For a simple example, we first return to the case $K(x) = \frac{1}{2}x^2$ and $D = [0,\infty)$ from Section \ref{subsect:escape}, adding a potential $V(x) = gx$. In Theorem \ref{existence1} we then establish a condition on $V$ (see \eqref{Vbound1}) that is in some sense minimal and guarantees existence of a global minimizer of $\CalE^\nu$ for general domains $D$ satisfying Assumption 2. Finally, Theorem \ref{existence2} provides a weaker set of requirements on $V$ which takes advantage of symmetries within the domain.
\begin{thm}\label{COM2}
Let $D = [0, +\infty)$, $K(x) = \frac{1}{2}x^2$ and $V(x) = gx$ for $g>0$. Then for any $g>0$, there exists a unique critical point $\gmin$ for $\CalE^\nu$ in the space of measures in $\CalP_2^{ac}(D)$ having support equal to $D$.
\end{thm}
\begin{proof}
\noindent The energy is given by
\begin{equation}\label{en-g}
\CalE^\nu[\rho] = \frac{1}{4}\int_0^\infty\int_0^\infty (x-y)^2\,\rho(x)\rho(y)\,dx\,dy+\nu\int_0^\infty\rho(x)\log(\rho(x))\,dx+g\int_0^\infty x\rho(x)\,dx.
\end{equation}
We proceed as in Theorem \ref{COM1} and look for energy minimizers using the fixed-point characterization of critical points \eqref{gibbstat}, only now we show that the map $T(\rho)$ has a unique fixed point. Borrowing from the calculations in Theorem \ref{COM1}, for any $\rho\in\CalP^{ac}_2(D)$ we have:
\[T(\rho) = Z^{-1}\exp\left(-\frac{K*\rho(x)+V(x)}{\nu}\right) = A(c)\,\exp\left(-\frac{1}{2\nu}\left(x-c\right)^2\right),\] 
where now $c = M_1(\rho)-g$. Since $T$ maps $\CalP^{ac}_2(D)$ into $\Gamma_c$, by Corollary \ref{fixedpointmap} it suffices to look for critical points in $\Gamma_c$. 

We then proceed as above and first attempt to satisfy the necessary condition $\frac{d}{dc}\CalE^\nu[\rho_c] = 0$. With $g>0$ the energy \eqref{COMenergy1} becomes
\begin{equation}\label{energy_quad_attract_formula}
\CalE^\nu[\rho_c] = \nu\log\left(\frac{2}{\sqrt{2\pi\nu}}\right) - \nu\left(f({\tilde{c}})+\frac{1}{4}f'({\tilde{c}})^2\right) + g\left(\sqrt{\frac{\nu}{2}} f'({\tilde{c}}) +\sqrt{2\nu}\,{\tilde{c}}\right),
\end{equation}
whereby solving $\frac{d}{dc}\CalE^\nu[\rho_c] = 0$ reduces to finding a root $\tilde{c}$ to
\[\left(\sqrt{\frac{2}{\nu}}g - f'(\tilde{c})\right)\left(1+\frac{1}{2}f''(\tilde{c})\right) = 0.\]
From \eqref{COMcont}, we know that the second term is strictly positive, so we may divide by it and further reduce the problem to solving
\begin{equation}\label{COMroot} 
f'(\tilde{c}) = \sqrt{\frac{2}{\nu}}\,g.
\end{equation}
For any $g,\nu>0$, \eqref{COMroot} has a unique solution since $f':\R\to[0,\infty)$ is smooth and monotonically decreasing, so we have that there exists a unique \textit{candidate} critical point $\rho_{c^*}\in \Gamma_c$ where $\frac{c^*}{\sqrt{2\nu}}$ solves \eqref{COMroot}; $\rho_{c^*}$ is then a critical point of $\CalE^\nu$ over the space $\Gamma_c$.

All that remains is to show that $T(\rho_{c^*})=\rho_{c^*}$ to conclude that $\rho_{c^*}$ is in fact a critical point of $\CalE^\nu$ over all of $\CalP^{ac}_2(D)$. Indeed, since $T$ maps $\CalP^{ac}_2(D)$ into $\Gamma_c$, we have $T(\rho_{c^*})=\rho_{c'} \in \Gamma_c$ for some $c'\in\RR$, and by direct calculation,
\[c' = M_1(\rho_{c^*})-g = \sqrt{\frac{\nu}{2}}f'\left(\frac{c^*}{\sqrt{2\nu}}\right)+c^*-g = c^*,\]
since $\frac{c^*}{\sqrt{2\nu}}$ solves \eqref{COMroot}. This shows that $\rho_{c'} = \rho_{c^*}$ since every member of $\Gamma_c$ is uniquely determined by its shift $c$. This completes the proof. We refer the reader back to Figure \ref{energy_quad_attract} for a comparison of $c\to \CalE^\nu[\rho_c]$ for several values of $g>0$, with each plot showing one global minimum.
\end{proof}
\begin{rmrk}
\label{rmk:gc}
\normalfont
For $g=g_c := \sqrt{\frac{2\nu}{\pi}}$, the solution $c$ to \eqref{COMroot} is exactly $c=0$, which implies that the critical point $\gmin$ is exactly a half-Gaussian, and for $g\geq g_c$, the maximum of $\gmin$ lies at $x=0$. This is used to benchmark the numerical method in Section \ref{sect:numerics}, in Figure \ref{kp2}.
\end{rmrk}


\begin{thm}\label{existence1}
Suppose that Assumptions \ref{assumKV} and \ref{assumD} are satisfied and that $K$ and $V$ are positive. In addition, suppose that for some $\delta > 0$ and $C_K\in \RR$,
\vspace{5pt}\begin{equation}\label{existFree}\vspace{5pt}
K(x) \geq 2(1+\delta)d\nu \log|x| + C_K,
\end{equation}
and that for some $x_0\in D$, $V$ satisfies
\begin{equation}\label{Vbound1}
\lim_{R\to \infty}\left( \inf_{x\in B_R^c(x_0)} V(x)\right) = +\infty.
\end{equation}
Then there exists a global minimizer $\gmin \in \CalP^{ac}(D)$ of $\CalE^\nu$.
\end{thm}
\begin{proof}
We will first show that the energy $\CalE^\nu$ is bounded below and then prove that minimizing sequences are tight. Indeed, the boundedness from below of $\CalE^\nu$ follows from results in free space. By Theorem \ref{tightbound} above (along with \cite{carrillo2018existence}), relation \eqref{existFree} between $K$ and $\nu$ is sufficient to guarantee that $\CalE^\nu$ is bounded below over $\CalP(\Rd)$ by a constant $C\in \RR$ when $V=0$.

Since $|D|>0$ and we are not requiring any regularity of measures other than absolute continuity with respect to Lebesgue measure, for any $\mu \in \CalP^{ac}(D)$ with density $\rho$ we can define a measure $\mu_0\in \CalP^{ac}(\Rd)$ with density $\rho_0(x)$ by extending $\rho$ by zero:
\begin{equation}\label{rhozero}
\rho_0(x) = \begin{cases} \rho(x), & x\in D \\  0, & x \in D^{\,c}.\end{cases}
\end{equation}
For each $\mu \in \CalP^{ac}(D)$ we then have the lower bound
\begin{align*}
\CalE^\nu[\mu] &=\frac{1}{2}\int_D\int_DK(x-y)\,d\mu(x)d\mu(y)+\nu\int_D\rho(x)\log(\rho(x))\,dx+\int_DV(x)\,d\mu(x) \\
&= \frac{1}{2}\int_{\Rd}\int_{\Rd} K(x-y)\,d\mu_0(x)d\mu_0(y)+\nu\int_{\Rd}\rho_0(x)\log(\rho_0(x))\,dx+\int_{D} V(x)\,d\mu(x)\\
&>C + \int_{D} V(x)\,d\mu(x),
\end{align*}
which implies
\begin{equation}\label{Vbound2}
\int_DV(x)\,d\mu(x) < \CalE^\nu[\mu]-C.
\end{equation}
Now consider a minimizing sequence $\left\{\mu_n\right\}_{n\geq 0}\subset \CalP^{ac}(D)$ of $\CalE^\nu$. The following argument shows that $\left\{\mu_n\right\}_{n\geq 0}$ is tight. Since $\left\{\mu_n\right\}_{n\geq 0}$ is minimizing, we can assume $\left\{\CalE^\nu[\mu_n]\right\}_{n\geq 0}$ is bounded above, hence 
\eqref{Vbound2} implies
\begin{equation}\label{Vbound3}
\sup_n\int_D V(x)\,d\mu_n(x) < M
\end{equation}
for some $M\in \RR$. Fix $\epsilon>0$ and let $L>0$ be large enough that $M/L<\epsilon$. From \eqref{Vbound1}, to $L$ there corresponds an $R$ such that 
\[\inf_{x\in B^c_R(x_0)}V(x)>L.\]
For each $\mu_n$ we then have 
\[L\int_{B^c_R(x_0)\cap D}\,d\mu_n(x)\leq \int_{B^c_R(x_0)\cap D}V(x)\,d\mu_n(x)\leq \int_DV(x)\,d\mu_n(x)<M,\]
hence for the compact set $K_\epsilon = B_R(x_0)\cap D$,
\[\mu_n(K_\epsilon)> 1-\epsilon.\]
This shows that the minimizing sequence $\left\{\mu_n\right\}_{n\geq 0}$ is tight. By Prokhorov's theorem we may then extract a subsequence $\left\{\mu_{n_k}\right\}_{k\geq 0}$ which converges in the weak-* topology of measures to some $\gmin\in \CalP(D)$. It follows from the weak-* lower semicontinuity of $\CalE^\nu$ (Lemma \ref{lowerSC}) that 
\[\CalE^\nu[\gmin] \leq \liminf_{k\to \infty}\CalE^\nu[\rho_{n_k}] = \lim_{n\to \infty}\CalE^\nu[\rho_n] = \inf_{\rho \in \CalP(D)} \CalE^\nu[\rho],\]
and so $\gmin$ realizes the infimum. Moreover, by Remark \ref{ac-gmin}, $\gmin\in \CalP^{ac}(D)$.
\end{proof}
The previous theorem provides a way to guarantee existence of a minimizer in all domains $D$ satisfying Assumption 2, simply by adding an external potential to contain the mass and enforce tightness. As the following theorem shows, in many domains a less restrictive external potential is needed to ensure a minimizer.

We will need some terminology for the next theorem. Define a \textit{band} $S^i_a$ in $\Rd$ by 
\[S^i_a = \left\{x\in \Rd \,:\, |x_i|<a\right\}.\]
Also, we define a function $f:\Rd\to \R$ to be \textit{discrete-translation invariant in} $u \in \Rd$, if for any $m\in \ZZ$, 
\[f(x+m u) = f(x), \qquad \text{for all } x\in \Rd.\] 
A subset $D\subset \Rd$ is called discrete-translation invariant in $u \in \Rd$ if its indicator function $\ind{D}$ is discrete-translation invariant in $u$ by definition above. 
\begin{thm}\label{existence2}
Let $(x_1,\dots,x_d)$ be a fixed orthogonal coordinate system for $\Rd$. Suppose the hypotheses of Theorem \ref{existence1} are satisfied, except that \eqref{Vbound1} is replaced with the following: for each coordinate $x_i$, at least one of the following holds:
\begin{enumerate}[label=(\roman*)]
\item $D$ is bounded in $x_i$.
\item $V$ is unbounded in $x_i$ of the form
\begin{equation}\label{Vcoordbound}
\lim_{a\to \infty}\left( \inf_{x\in (S^i_a)^c} V(x)\right) = +\infty.
\end{equation}
\item $D$ and $V$ are discrete-translation invariant in $s_i\hat{e}_i$ for some $s_i>0$.
\end{enumerate}
Then there exists a global minimizer $\gmin\in \CalP^{ac}(D)$ of $\CalE^\nu$.
\end{thm}
\begin{proof}
As before, we consider a minimizing sequence $\left\{\mu_n\right\}_{n\geq 0} \subset \CalP^{ac}(D)$ for $\CalE^\nu$ over $\CalP(D)$, where we assume that $\left\{\CalE^\nu[\mu_n]\right\}_{n\geq 0}$ is bounded above by some $\tilde{M}>0$. Again, \eqref{existFree} implies that $\left\{\CalE^\nu[\mu_n]\right\}_{n\geq 0}$ is bounded below, and so the upper bound \eqref{Vbound3} on $\left\{\CalV[\mu_n]\right\}_{n\geq 0}$ still holds. As in \eqref{rhozero}, an absolutely continuous measure $\mu\in \CalP^{ac}(D)$ may be extended by zero to a probability measure on $\Rd$, so with some abuse of notation we will refer to $\mu\in \CalP^{ac}(\Rd)$ as a probability measure on $D$ whenever $\mu(D^c) = 0$.

Since we can no longer extract tightness just from $V$, we will instead exploit the fact that the interaction energy $\CalK$ is bounded and use Lemma \ref{Ktight} to conclude that $\left\{\mu_n\right\}_{n\geq 0}$ is tight-up-to-translations in $\CalP(\Rd)$. Then, with each coordinate $x_i$ satisfying at least (i), (ii) or (iii), we will show that a translated sequence $\left\{\tilde{\mu}_n\right\}_{n\geq 0}$ exists that lies in $\CalP(D)$ and remains minimizing. 

To see that the interaction portion of the energy is bounded, we reuse some arguments from \cite{carrillo2018existence}. Namely, the logarithmic HLS inequality (Lemma \ref{lemma:HLS}) together with \eqref{existFree} imply that for each $\mu\in \CalP^{ac}(D)$ with $d\mu(x) = \rho(x)\,dx$,
\begin{align*} 
\nu\CalS[\mu] &\geq -\nu d\int_{\Rd}\int_{\Rd}\log(|x-y|)\rho(x)\rho(y)\,dx\,dy-\nu dC_0\\
&\geq -\frac{1}{2(1+\delta)} \int_{\Rd}\int_{\Rd}K(x-y)\,d\mu(x)\,d\mu(y)-\nu dC_0-\frac{1}{2(1+\delta)}C_K\\
&= - \frac{1}{1+\delta}\CalK[\mu]-\tilde{C}
\end{align*}
for $\tilde{C} = \nu dC_0+\frac{1}{2(1+\delta)}C_K$. By the positivity of $V$, for each $\mu_n$ we have
\[\frac{\delta}{1+\delta}\CalK[\mu_n]\ \ \leq\ \ \CalK[\mu_n]+\nu\CalS[\mu_n]+\CalV[\mu_n] + \tilde{C}\ \ =\ \  \CalE^\nu[\mu_n]+\tilde{C}\ <\ \tilde{M}+\tilde{C},\]
hence $\left\{\CalK[\mu_n]\right\}_{n\geq 0}$ is bounded above. By Lemma \ref{Ktight} we now have that $\left\{\mu_n\right\}_{n\geq 0}$ is tight up to translations in free space.

We now construct a tight, translated version of $\left\{\mu_n\right\}_{n\geq 0}$ that retains the property $\mu_n(D)=1$ and remains energy minimizing. To do so we address each coordinate $x_i$ and consider the three cases above. Let $\epsilon>0$ be given. \\

\noindent (i) For each $x_i$ in which $D$ is bounded, let $L_i = \displaystyle\sup_{x\in D}|x_i|$ and note that for each $n$
\[\mu_n(S^i_{L_i}) = 1 > 1-\epsilon.\] 
\smallskip
\noindent (ii) Similarly, for each $x_i$ in which $V$ satisfies \eqref{Vcoordbound}, there exists $L_i>0$ such that
\[\mu_n(S^i_{L_i})>1-\epsilon\]
uniformly in $n$ by a similar argument as in Theorem \ref{existence1}. Indeed, since $\left\{\CalV[\mu_n]\right\}_{n\geq 0}$ is bounded above by some $M>0$, let $L$ be large enough that $M/L<\epsilon$. Then there exists $L_i>0$ such that
 \[\inf_{x\in  \bigl(S^i_{L_i}\bigr)^c}V(x)>L,\]
hence 
\[\int_{\bigl(S^i_{L_i}\bigr)^c\cap D}\,d\mu_n(x)\leq \frac{1}{L}\int_{\bigl(S^i_{L_i}\bigr)^c\cap D}V(x)\,d\mu_n(x)\leq \frac{1}{L}\int_DV(x)\,d\mu_n(x)< \epsilon.\]
\smallskip
\noindent (iii) Now consider the index set $I$ of coordinates $x_i$ for which $D$ and $V$ are discrete translation invariant in $s_i\hat{e}_i$ for some $s_i>0$. First we note that if $D$ is discrete translation invariant, then so are $\CalK$ and $\CalS$ by a change of variables. If $V$ is also discrete translation invariant, then so is $\CalV$, hence for each $i\in I$, the energy $\CalE^\nu$ is discrete translation invariant in $s_i\hat{e}_i$.

Let $\left\{\mu^1_n\right\}_{n\geq 0} = \left\{\mu_n(x-x^{n,1})\right\}_{n\geq 0} \subset \CalP^{ac}(\Rd)$ be a translated sequence which is tight but may no longer satisfy $\mu^1_n(D)=1$. Without loss of generality we have $x^{n,1}_i = 0$ for $i\notin I$ using the arguments above for (i) and (ii), so translations have only occurred in coordinates $x_i$ for $i\in I$.

Now define another translated sequence $\left\{\mu^2_n\right\}_{n\geq 0}$ by
\[\mu^2_n := \mu^1_n(x+\tilde{x}^{n,1}) = \mu_n(x-x^{n,2})\]
where the translations are defined by
\[x^{n,2} := x^{n,1} - \sum_{\substack{i=1\\ i\in I}}^d\text{mod}\left(x^{n,1}_i,s_i\right)\hat{e}_i :=  x^{n,1}-\tilde{x}^{n,1}\]
where 
\[\text{mod}\left(x^{n,1}_i,s_i\right) := x^{n,1}_i - \Big\lfloor\frac{x^{n,1}_i}{s_i}\Big\rfloor s_i.\]
From this we get for each $i\in I$ that
\[x^{n,2}_i = \Big\lfloor\frac{x^{n,1}_i}{s_i}\Big\rfloor s_i = m^n_i s_i, \qquad \text{for some } m^n_i\in \ZZ.\]
Hence by discrete translation invariance,
\[\mu^2_n(D) = \mu_n\left(D - x^{n,2}\right) = \mu_n\left(D - \sum_{\substack{i=1\\ i\in I}}^dm^n_i s_i\hat{e}_i\right) = \mu_n(D)=1,\]
and so $\left\{\mu^2_n\right\}_{n\geq 0}$ lies in $\CalP(D)$. Similarly,
\[\CalE^\nu[\mu^2_n] = \CalE^\nu\left[\mu_n\left(x - \sum_{\substack{i=1\\ i\in I}}^dm_i s_i\hat{e}_i\right)\right] = \CalE^\nu[\mu_n],\]
thus $\left\{\mu^2_n\right\}_{n\geq 0}$ retains the minimizing property of the original sequence $\left\{\mu_n\right\}_{n\geq 0}$. To see that $\left\{\mu^2_n\right\}_{n\geq 0}$ is tight, we can use the fact that $\left\{\mu^1_n\right\}_{n\geq 0}$ is tight to find a compact set $K^1_\epsilon\subset \Rd$ for which $\mu^1_n(K^1_\epsilon)>1-\epsilon$ for each $n$. Since 
\[\left\vert x^{n,2} - x^{n,1}\right\vert \leq \sqrt{d}\max_{i\in I}s_i,\]
the compact set 
\[K^2_\epsilon = \left\{x\in D \,:\, \dist{x}{K^1_\epsilon}\leq \sqrt{d}\max_{i\in I}s_i\right\}\]
satisfies $\mu_n^2(K^2_\epsilon)>1-\epsilon$ for each $n$. We may now apply Prokhorov's theorem and lower semicontinuity of the energy to extract a convergent subsequence $\left\{\mu^2_{n_k}\right\}_{k\geq 0}$ such that $\mu^2_{n_k}\overset{*}{\rightharpoonup} \gmin \in \CalP(D)$ and $\CalE^\nu[\gmin] = \inf\CalE^\nu>-\infty$. Finally, as above, Remark \ref{ac-gmin} implies that $\gmin\in \CalP^{ac}(D)$, which completes the proof.
\end{proof}
\begin{rmrk} 
\normalfont Theorem \ref{existence2}, although more technical, is designed to capture many practical cases. As it reads, one such case is that of the half-space domain $D = \RR^{d-1}\times [0, \infty)$ together with a potential $V$ of the form
\[V(x) = V(x_n) \leq Cx_n^p,\] 
for $p>0$, which only depends on the final coordinate $x_n$. This is the case commonly considered when modeling a swarm in a gravitational field. Another case is that of an infinite channel $D = B_R^{d-1}\times\RR$ where $B_R^{d-1}$ is a $(d-1)$ dimensional ball of radius $R$. Since the infinite channel is either bounded or translation invariant in each coordinate, a global minimizer exists for $V=0$, consistent with Theorem \ref{tightbound}.
\end{rmrk}
%
%

\section{Numerical Computation of Critical Points}
\label{sect:numerics}

\hspace{5mm} We compute critical points of $\CalE^\nu$ under purely attractive power-law potentials and attractive-repulsive potentials on an interval $D=  [0,L]$, using an iterative method to find fixed points of \eqref{gibbstat}. In all cases, we check that the Euler-Lagrange equation \eqref{ELcond} is satisfied to within the error tolerance of the iterative method. Due to the exponential decay of critical points it can be assumed that for sufficiently large $L$, critical points computed on the interval $D = [0,L]$ are good approximations of critical points on $[0,\infty)$ (when the latter exist). In light of Theorem \ref{domassymthm}, which implies that for $V=0$ no critical points exist on the half-line, computations with $V=0$ should be interpreted as approximations to critical points in free space ($D=\RR$), while computations made with $V\neq 0$ should be interpreted as approximations to critical points on the half line $D = [0,\infty)$.


\subsection{Numerical Method}
\paragraph{Fixed-Point Iterator.} The following scheme computes critical points of $\CalE^\nu$ by discretizing the map $\fxptmap:\CalP(D)\to\CalP^{ac}(D)$ given in \eqref{gibbstat}. Recall that 
%
%
fixed points of $\fxptmap$ are critical points of $\CalE^\nu$ (in particular, the set of fixed points of $\fxptmap$ are exactly the critical points of $\CalE^\nu$ which are absolutely continuous and supported on the whole domain). We use the iterative scheme
\begin{equation}\label{fxptscheme}
\rho^{n+1} = (1-\fxpt_n)\rho^n+\fxpt_n\fxptmap(\rho^n),
\end{equation}
where
\begin{equation}
\fxpt_n =\begin{cases} 1, &\text{if }\CalE^\nu\left[\fxptmap(\rho^n)\right]<\CalE^\nu[\rho^n],\\[2pt] \fxpt_c,&\text{otherwise},\end{cases}
\end{equation}
with inputs $\fxpt_c \in (0,1)$ and $\rho^0\in \CalP^{ac}(D)$.

In words, each iteration produces an absolutely continuous probability measure $\rho^{n+1}$ that is a convex combination of the previous iterate $\rho^n$ and its image under $\fxptmap$, unless the energy of $\fxptmap(\rho^n)$ is lower than that of $\rho^n$, in which case $\rho^{n+1} = \fxptmap(\rho^n)$. Each step requires computation of the integral terms in $\fxptmap(\rho^n)$ and $\CalE^\nu[\rho^n]$, which for $D=[0,L]$ is done by discretizing the interval into $N$ quadrature nodes and numerically integrating. For uniform grids, we use MATLAB's \texttt{conv} function to compute $K*\rho^n$, while for non-uniform grids we use trapezoidal integration. The scheme is terminated when 
\begin{equation}\label{fxptconv}
\nrm{\rho^n-\fxptmap(\rho^n)}_{L^1(D)} < \txt{0.05}{tol} \qquad \text{or} \qquad n>N_{\max},
\end{equation}
where tol and $N_{\max}$ are specified by the user. In what follows, we denote by $\rhofp$ the numerical solution produced by the fixed-point iterator upon convergence.


\paragraph{Stability Constraints.} The scheme \eqref{fxptscheme} has many benefits. It is explicit, so only numerical integration is required at each step, which allows for flexibility of the spatial grid. It is also positivity preserving. Due to the explicit nature, however, there are a few stability constraints. 
\medskip

{\em Oscillations.} The first stability constraint prevents spurious oscillations and can be explained by casting the scheme as a discretization of the following integro-differential equation: assuming $\fxpt_n \ll 1$, \eqref{fxptscheme} can be viewed as a forward-Euler discretization of\\[-10pt]
\begin{equation}\label{diffmap}
\begin{dcases} 
\ppt \rho(x,t) = \fxptmap(\rho(x,t))-\rho(x,t), & (x,t)\in D\times (0,\infty),\\
\hspace{3.75mm}\rho(x,0) = \rho_0(x) \in \CalP^{ac}(D), & x\in D,\\[1pt]
\end{dcases}
\end{equation}
whose steady states are exactly the fixed points of $\fxptmap$. 

For any point $x^*\in D$, the time evolution of $\rho(x^*,t)$ under \eqref{diffmap} is such that $\rho(x^*,t)$ increases when $\rho(x^*,t)<\fxptmap(\rho(x^*,t))$ and decreases when $\rho(x^*,t)>\fxptmap(\rho(x^*,t))$. Analytically, if $\rho_0$ lies in the basin of attraction of some fixed point $\gmin$ of $T$, we expect pointwise convergence $\lim_{t\to \infty} \rho(x^*,t)=\gmin(x^*)$. If $\rho(x^*,t)$ oscillates around $\gmin(x^*)$ as it approaches $\gmin(x^*)$, numerically one can expect spurious growth of such oscillates. Indeed, oscillations do appear in the fixed-point method \eqref{fxptscheme} for ``timesteps'' $\fxpt_c$ that are too large, in which case the iterates $\rho^n$ cycle indefinitely through a finite set of measures.

To arrive at a stable value of $\fxpt_c$ which prevents oscillations, we examine a bound on the $L^1$-Lipschitz constant $L_{\fxptmap}$ of $\fxptmap$ (derived in \cite{messenger2019aggregation} assuming $D$ is bounded, $K$ is bounded below, and $V$ is positive):
%
\begin{equation}\label{lipboundfp}
L_{\fxptmap} \leq\frac{2}{\nu} \nrm{\tilde{K}}_{L^\infty(D-D)} \exp\left(\frac{1}{\nu}\nrm{\tilde{K}}_{L^\infty(D-D)}\right),
\end{equation}
where $\tilde{K}:=K - \min_{x\in D-D}\left\{K(x)\right\}$. Due to the exponential dependence on $\nrm{K}_{\infty}$, \eqref{lipboundfp} may not be a very encouraging bound, but it does suggest that $\fxpt_c$ should be proportional to $\nu$. Indeed, we see convergence of the scheme for $\fxpt_c = \mathcal{O}(\nu)$ and in all computations below  set $\fxpt_c = 5\nu$. Direct dependence of $\fxpt_c$ on $\nrm{K}_{\infty}$ was not observed. 

\medskip

{\em Normalization and Underflow.} Another numerical issue is round-off error. Assuming for the moment that $K$ and $V$ are both positive, when $\nu$ is small the argument of the exponent in $\fxptmap$ is negative and large in magnitude. This results in underflow of digits when calculating $Z(\rho)$ and subsequent division by a small quantity. To avoid this, we exploit the fact that the set of critical points of $\CalE^\nu$ is unchanged by adding a constant to $K$ and at each step normalize the argument of the exponent by adding to $K$ the factor $c_n := - \nu \log Z(\rho^{n-1})$. The potential used in simulations then changes at each iteration and is given by $K_n(x) = K_{n-1}(x)+c_n$ with $K_0 = K$. For $Z(\rho^n)$ we then have
\[Z(\rho^n) = Z(\rho^{n-1}) \int_D \exp\left(\displaystyle -\frac{K_{n-1}*\rho^n(x)+V}{\nu}\right) dx,\]
and so as $\rho^n\to \gmin$ we see that $Z(\rho^n)\to 1$. This normalization turns out to stabilize the problem, and results in the constant on the right-hand side of the Euler-Lagrange equation \eqref{ELcond} conveniently converging to zero, since the true value $\lambda$ is equal to $-\nu\log Z(\gmin)$.

\paragraph{Continuation.} The scaling $\fxpt_c = \mathcal{O}(\nu)$ adopted in light of the bound \eqref{lipboundfp} implies that the number of iterations required for convergence is large for small $\nu$. To prevent this, we use continuation on $\nu$: we let $\left\{\nu_j\right\}_{j=0}^J$ be a decreasing sequence of diffusion parameters and for each $\nu_j$ we set the initial guess $\rho^0_j$ for the fixed-point iterator to the output $\rhofp^{j-1}$ of the fixed-point iterator with $\nu=\nu_{j-1}$. This is very effective for reducing the total iteration count and is also seen in Figure \ref{multistate} to be crucial for revealing non-uniqueness of critical points: different sequences of diffusion parameters (sharing the same final value of $\nu$) can produce different critical points.


\paragraph{Convergence Criteria.} We are primarily concerned with satisfying the Euler-Lagrange equation in its original form,
\[\Lambda(x) := K*\gmin(x) + \nu \log(\gmin(x)) + V(x) = \lambda, \qquad \text{ for all } x\in [0, L],\] 
where $\lambda = \CalE^\nu[\gmin]+\CalK[\gmin]$, and so we check that the quantity
\begin{equation}\label{laminf}
\Lambda_\infty = \nrm{\Lambda - \CalE^\nu[\rhofp] - \CalK[\rhofp]}_\infty
\end{equation}
is below the chosen error tolerance for each numerical solution $\rhofp$. We also check that the boundary condition \eqref{boundaryV} derived in Theorem \ref{domassymthm} is satisfied, which reads
\[\gmin(0)-\gmin(L) = \frac{1}{\nu}\int_0^L V'(x)\gmin(x)\,dx.\]
However, in all numerical experiments we use $V(x) = gx$ and choose $L$ large enough that $\gmin(L)$ is negligible, so this reduces to
\begin{equation}\label{reducedbc}
\gmin(0) = g/\nu\ ,
\end{equation}
which is exact for $D=[0,\infty)$. Thus, we also assess the relative error
\begin{equation}\label{bccheck}
\bcerr := \frac{\left\vert \rhofp(0) - g/\nu\right\vert}{g/\nu}.
\end{equation}
%


\subsection{Purely Attractive Interaction Potential}
\label{subsect:attraction}
The first class of potentials we examine are purely attractive, power-law potentials 
\[K_p(x) := \frac{1}{p}|x|^p,\]
for $p > 0$, where repulsive forces are present in the swarm only in the form of diffusion. Without diffusion, for all $p >0$ the global minimizer is a single $\delta$-aggregation with location determined by $V$. The effect of switching on diffusion is to smooth out the $\delta$-aggregation. Indeed, Figures \ref{kp2}--\ref{kplarge} show critical points which are continuous and unimodal, but are supported on the whole domain with fast-decaying tails.

First we examine the case $p=2$ in detail given the results in Theorem \ref{COM2}, and compare with other small values of $p$. Then we look into the limit of large $p$, which is motivated by the fact that minimizers of $\CalE^\nu$ are supported on the entire domain regardless of the attraction strength (see Theorem \ref{suppmin}).
\begin{rmrk}
\normalfont
For uniformly convex interaction potentials (i.e $K_p$ with $p\geq 2$), it can be shown using displacement convexity as in \cite{mccann1997convexity} and \cite{carrillo2003kinetic} that for $V= gx$ the global minimizer is the unique critical point of $\CalE^\nu$ for $D=[0,\infty)$. Numerics suggest that uniqueness holds for general power-law, purely-attractive $K$ when (i) $D = \RR$ with $V=0$ and (ii) $D = [0,+\infty)$ with $V$ convex and satisfying condition \eqref{Vbound1}. In this way, convexity of $K$ may be relaxed if $K$ remains purely attractive. However, Figure \ref{multistate} below shows that for attractive-repulsive (non-convex) interaction potentials, minimizers are not unique.
\end{rmrk}
\paragraph{Moderate Attraction Strength and Connection to Theorem \ref{COM2}.} The case $p=2$ is used to benchmark the fixed-point iterator. Theorems \ref{COM2} and \ref{existence1} together imply that for every $g>0$ and $\nu>0$ the global minimizer of $\CalE^\nu$ under $K_2$ and $V=gx$ is the unique critical point; moreover, we have an explicit formula for the global minimizer, up to solving equation \eqref{COMroot} (e.g. with MATLAB's \texttt{fzero} command). Figure \ref{kp2} contains computed solutions under $K_2$ for several values of $g$ along with convergence data. Agreement with the exact solution $\gmin_{_{ex}}$ and the Euler-Lagrange equation as measured by $\Lambda_\infty$ and the boundary condition $\bcerr$ are all on the order of the chosen error tolerance of $\sci{1}{6}$. We see especially good agreement with the Euler-Lagrange equation, gaining two orders of accuracy relative to the error tolerance.
\begin{figure}[htp]
\centering
\begin{tabular}{cc}
        \includegraphics[trim={30 0 30 20},clip,width=0.7\textwidth]{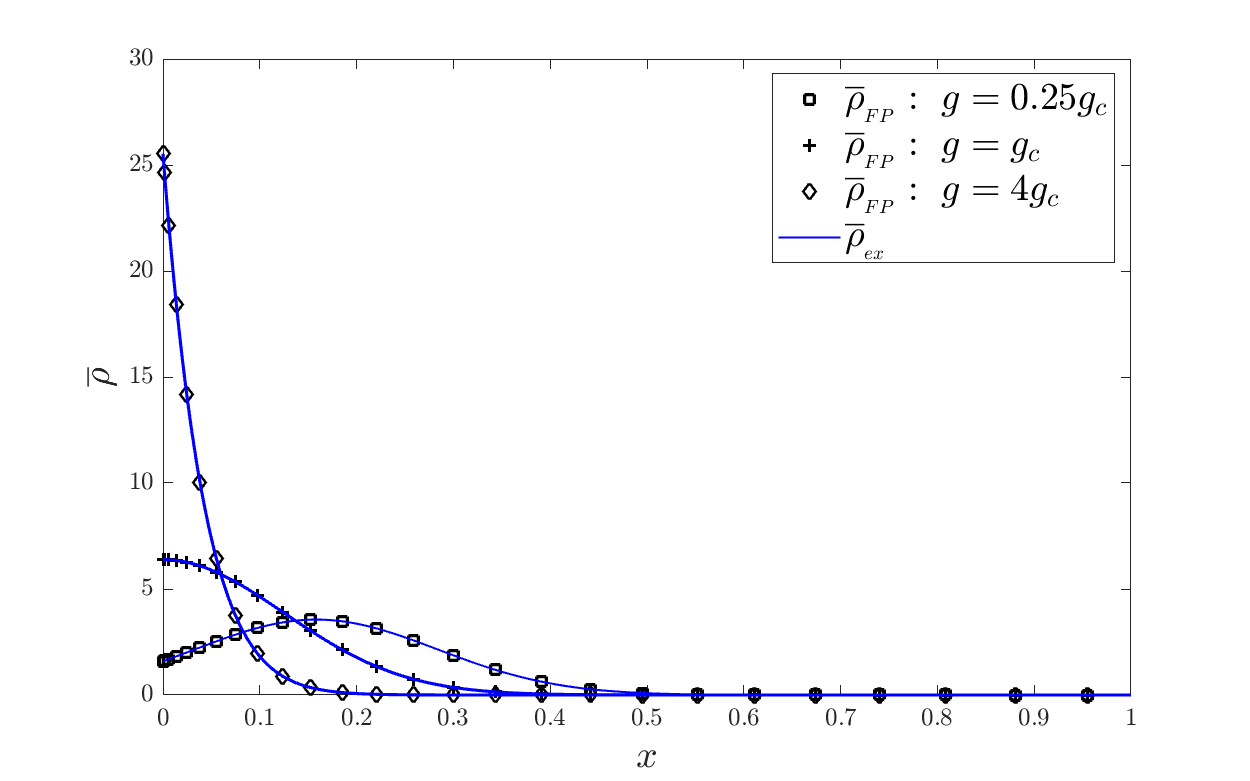} \\
\begin{tabular}{| c | c | c | c | c |}
\hline
$g$ & $\nrm{\rhofp - \gmin_{_{ex}}}_1$ & $\Lambda_\infty$ & $\bcerr$ & Total Iterations \\ \hline
$0.25 g_c$ & $\sci{3.92}{6}$ & $\sci{9.54}{8}$ & $\sci{7.91}{6}$ & 44\\
$g_c$ & $\sci{2.16}{6}$ & $\sci{6.63}{8}$ & $\sci{7.38}{7}$ &  16\\
$4g_c$ & $\sci{8.13}{6}$ &$\sci{3.88}{9}$ & $\sci{7.40}{6}$ &  10 \\ \hline
\end{tabular}\\
\end{tabular}
\caption[Global minimizers of $\CalE^\nu$ computed using fixed-point iteration for $K_p$ with $p=2$, $V(x) = gx$ and $\nu = 2^{-6}$]{Global minimizers under $K_p$ with $p=2$, $V(x) = gx$ and $\nu = 2^{-6} \approx 0.0156$ computed using the fixed-point iterator. The method is initialized at $\rho^0 = 4\ind{[0,0.25]}$ with an error tolerance of $\sci{1}{6}$. A spatial grid of $N=2^{10}$ points is used with points spaced quadratically to resolve the boundary at $x=0$ (not all points are plotted). The value $g_c := \sqrt{2\nu/\pi}$ is emphasized because solutions achieve their maximum at $x=0$ if and only if $g\geq g_c$ (see Remark \ref{rmk:gc}). All three computed solutions converged in well under $N_{\max}=2000$ iterations.}
\label{kp2}
\end{figure}
Figure \ref{kpsmall} shows computed critical points for $p\in(1,8]$ to compare with the case $p=2$, demonstrating that increasing $p$ decreases the maximum height of the solution. We examine the cases $g=0$ and $g=\nu$, the former resulting in critical points which are symmetric about the center of the domain while the latter causes clustering near the domain boundary at $x=0$. Notice that with $g=\nu$, \eqref{reducedbc} implies that $\gmin(0) = 1$, which is clearly represented on the rightmost plot.


\begin{figure}[htp]
\centering
\begin{tabular}{cc}
\includegraphics[trim={10 0 20 20},clip,width=0.48\textwidth]{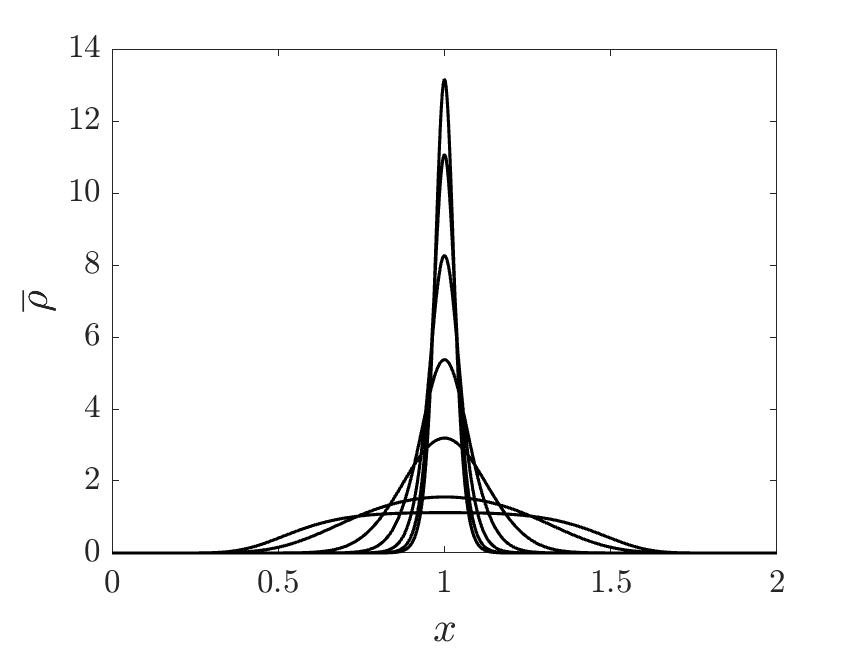} &
\includegraphics[trim={10 0 20 20},clip,width=0.48\textwidth]{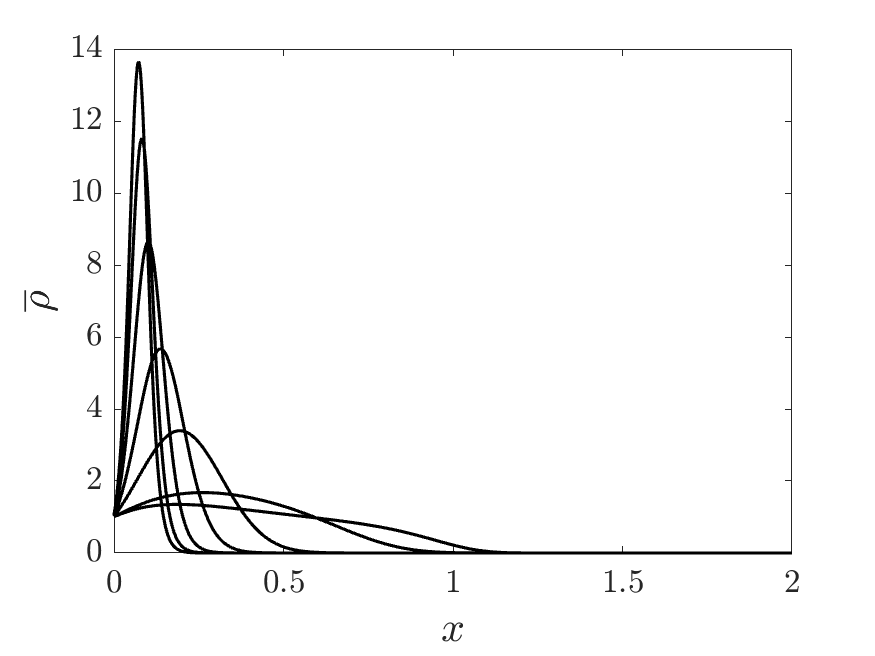} \\
\end{tabular}
\caption{Critical points of $\CalE^\nu$ computed using the fixed-point iterator with $K_p$ for $p\in\left\{1.0625,1.125,1.25,1.5,2,4,8\right\}$, $V(x) = gx$ and $\nu = 2^{-6}$. Left: profiles for $g=0$. Right: profiles for $g=\nu$. As $p$ increases, the maximum height of solutions decreases. The method is initialized at $\rho^0 = 0.5\ind{[0, 2]}$ for $g=0$ and $\rho^0 = \ind{[0, 1]}$ for $g=\nu$. The error tolerance is set to $\sci{1}{6}$ and the maximum iterations set to $N_{\max} = 2000$. A spatial grid of $N=2^{10}$ uniformly spaced points is used. In each case $\Lambda_\infty$ is well below $\sci{1}{6}$; however, the use of uniformly spaced points instead of quadratically spaced (as in Figure \ref{kp2}) has an effect on accuracy at the boundary: $E_0$ remains on the order of $10^{-2}$.}
\label{kpsmall}
\end{figure}


\paragraph{Limit of Large Attraction.} We now examine numerically the limit of large $p$, which is motivated by the fact that minimizers $\critrho$ of $\CalE^\nu$ satisfy $\supp{\critrho} =D$ regardless of how strong the (power-law) attraction is (see Theorem \ref{suppmin}). This is a striking feature because intuitively one might expect that for very large attraction the swarm would be confined to a compact set. Only as $p\to+\infty$, however, do we reach a state with compact support. We derive this family of compactly supported states below in one dimension and compute critical points for powers up to $p=256$ to suggest convergence to the compactly supported states included in Figure \ref{kplarge}.

The limit as $p\to \infty$ is clearly singular, as the limiting interaction potential $K_\infty$ defined by 
\[\lim_{p\to \infty} K_p(x) = K_\infty(x) := \begin{cases} 0, & x\in [-1,1]\\+\infty, & x\notin [-1,1]\end{cases}\]
is no longer locally integrable. As such, the space of probability measures on which the resulting energy is finite is very limited. Despite this, we can still determine minimizers for $\CalE^\nu$ under $K_\infty$. It is not hard to show that the corresponding interaction energy $\CalK_\infty$ satisfies
%
\[\CalK_\infty[\mu] = \begin{dcases} 0, & \txt{0.1}{if}\mu\txt{0.1}{is supported on a unit interval,} \\ +\infty, & \text{otherwise,}\end{dcases}\]
and so the space we should be looking for minimizers in is 
\[\left\{\mu\in \CalP_\infty(D) \txt{0.2}{:} \supp{\mu}\subset [a,1+a]\txt{0.2}{for some} a\in \RR\right\}.\]
 To arrive at this, for the interaction energy we have
\[\CalK_\infty[\mu] = \frac{1}{2}\int_D\int_DK_\infty(x-y)\,d\mu(y)\,d\mu(x) = \frac{1}{2}\int_{\supp{\mu}}
 K_\infty*\mu(x) \,d\mu(x),\]
which is finite if and only if $K_\infty*\mu$ is finite $\mu$-a.e. By computing
\begin{align*}
K_\infty* \mu(x) &= \int_{\supp{\mu}}K_\infty(x-y)\,d\mu(y)\\
&= \int_{\supp{\mu}\cap [x-1,x+1]^c }K_\infty(x-y)\,d\mu(y)\\
&= \begin{cases} 0, &\mu([x-1,x+1]^c)=0\\ +\infty, & \text{otherwise},\end{cases}
\end{align*}
we see that $\CalK_\infty[\mu] = +\infty$ unless $\mu([x-1,x+1]^c)=0$ for $\mu$-a.e. $x\in D$, which is equivalent to $\mu$ having support on a unit interval. From this we deduce that a minimizer $\critrho_\infty$ has support on a unit interval and satisfies $K_\infty* \critrho_\infty(x) = 0$ for $\critrho_\infty$-a.e.\ $x$. Hence the Euler-Lagrange equation reads
\[\nu\log(\critrho_\infty(x))+V(x) = \lambda, \qquad \critrho_\infty\text{-a.e. }x\in D,\]
or, taking $\supp{\critrho_\infty} = [0,1]$,
\begin{equation}\label{pinf}
\gmin_\infty = \begin{dcases} \ind{[0,1]} & \txt{0.2}{for} V=0 \\[5pt] Z^{-1} e^{-V/\nu}\ind{[0,1]} & \txt{0.2}{for} V\neq 0.\end{dcases}
\end{equation}

Figure \ref{kplarge} shows critical points for $K_p$ and $V=gx$ for larger values of $p$  together with the corresponding limiting measure $\critrho_\infty$ derived above. For $g=0$, as $p$ increases we see solutions increasing to $\critrho_\infty$ inside $[0.5,1.5]$ and decreasing to zero elsewhere. For $g=\nu$ the boundary condition \eqref{reducedbc} again reduces to $\gmin(0)=1$, which is satisfied through increasingly sharp transitions as $p$ increases, and is not satisfied in the limit by $\critrho_\infty$. We still see $\Lambda_\infty$ values near the error tolerance, except for $p=256$, where the method clearly breaks down, as the scheme converges in fewer than $N_{\max}$ iterations yet $\Lambda_\infty$ is $\mathcal{O}(1)$.

\begin{figure}[htp]
\centering
\begin{tabular}{cc}
\includegraphics[trim={0 0 20 20},clip,width=0.48\textwidth]{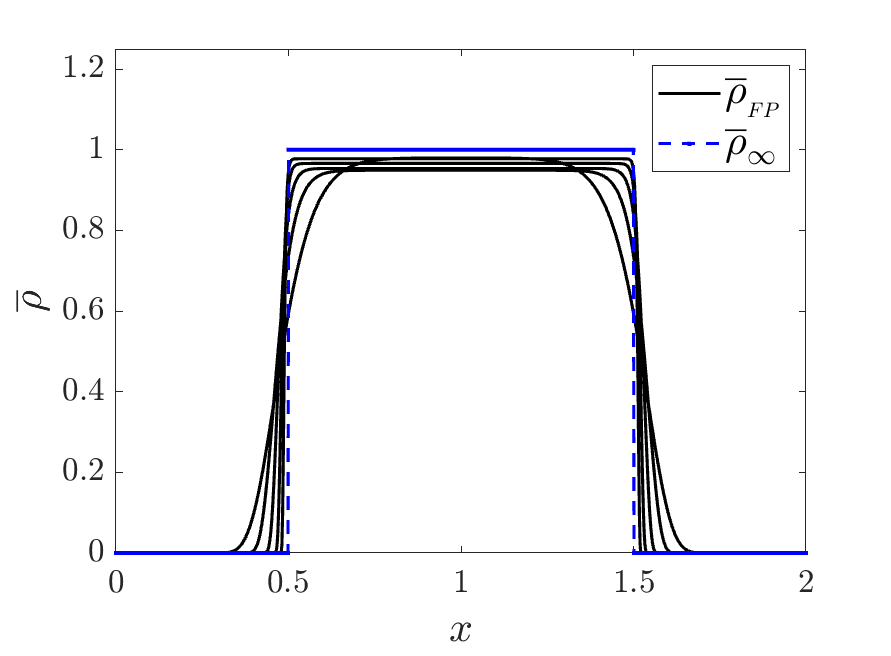} &
\includegraphics[trim={0 0 20 20},clip,width=0.48\textwidth]{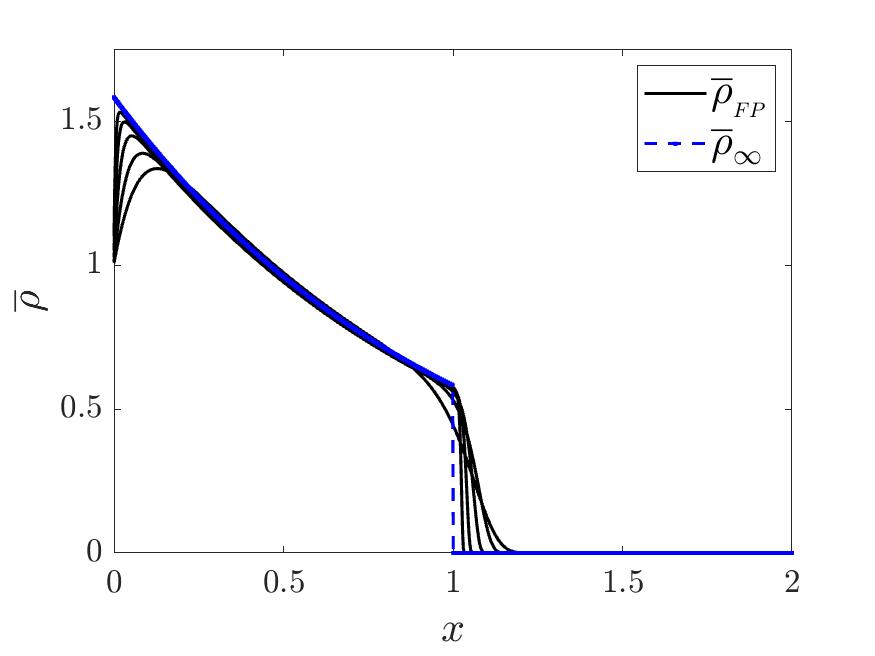} \\
\begin{tabular}{| c | c | c |}
\hline
$p$ & $\Lambda_\infty$ & Iterations \\ \hline
$16$ & $\sci{1.92}{6}$ & 64\\
$32$ & $\sci{3.84}{6}$ &  131\\
$64$ & $\sci{8.50}{6}$ &  168 \\
$128$ & $\sci{1.57}{5}$ & 186\\
$256$ & $\sci{2.33}{1}$ &  180\\ \hline
\end{tabular}
&
\begin{tabular}{| c | c | c | c |}
\hline
$p$ & $\Lambda_\infty$ & $E_0$ & Iterations \\ \hline
$16$ & $\sci{5.11}{6}$ & $\sci{7.16}{3}$ & 16\\
$32$ & $\sci{6.86}{6}$ & $\sci{1.29}{2}$ &  116\\
$64$ & $\sci{1.30}{5}$ & $\sci{2.55}{2}$ &  115 \\
$128$ & $\sci{2.70}{5}$ & $\sci{5.10}{2}$ & 136\\
$256$ & $\sci{2.48}{1}$ &  $\sci{9.97}{2}$ &  140\\ \hline
\end{tabular}
\end{tabular}
\caption[Critical points of $\CalE^\nu$ computed using fixed-point iteration for $K_p$ with $p =2^{k}$, $k=4, \dots, 8$, $V(x) = gx$ and $\nu = 2^{-6}$]{Critical points of $\CalE^\nu$ computed using the fixed-point iterator with $K_p$ for $p =2^{k}$, $k=4, \dots, 8$, $V(x) = gx$ and $\nu = 2^{-6}$. Left: profiles for $g=0$. Right: profiles for $g=\nu$. As $p$ increases, $\rhofp$ drops off sharply outside an interval of length 1, approaching the compactly supported state $\gmin_\infty$ defined in \eqref{pinf}. Parameters for the fixed-point iterator are the same as in Figure \ref{kpsmall}.}
\label{kplarge}
\end{figure}



\subsection{Non-Uniqueness under Attractive-Repulsive Potentials}
\label{subsect:att-rep}
The second class of interaction potentials we consider involve attraction at large distances and repulsion at short distances. So-called attractive-repulsive potentials have been the subject of a substantial amount of research in recent years (see \cite{balague2013dimensionality,bernoff2011primer,evers2016metastable,fellner2010stable,fetecau2017swarm,fetecau2017swarming,mogilner2003mutual})  for their use in modelling biological swarms, which predominantly seem to obey the following basic rules: if two individuals are too close, increase their distance, if too far away, decrease their distance. We show here through a numerical example that for such potentials, uniqueness of critical points does not hold in general.

We examine a regularization of the potential 
\[K_{QANR}(x) = \frac{1}{2}|x|^2+2\phi(x),\]
which features quadratic attraction and Newtonian repulsion given by the free-space Green's function $\phi(x) = -\frac{1}{2}|x|$ for the negative Laplacian $-\Delta$ in one dimension.
%
Specifically, we consider the $C^1$ regularized versions of $K_{QANR}$ in the form of the one-parameter family
\begin{equation}
K_\epsilon(x) := \frac{1}{2}x^2 + 2\phi_\epsilon(x):= \frac{1}{2}x^2 + \begin{dcases} -|x|, & \qquad |x|>\epsilon, \\ -\frac{\epsilon}{2}- \frac{1}{2\epsilon}x^2, &\qquad  |x|\leq \epsilon,\end{dcases}
\end{equation}
for $\epsilon\in (0, 1]$.

One might expect that for each $\epsilon$, switching on diffusion selects a unique number of aggregates in all minimizing states. Similar results have been documented: Evers and Kolokolnikov establish in \cite{evers2016metastable} that adding any level of diffusion to an equilibrium consisting of two aggregates of unequal mass for the plain aggregation model under the double-well potential $K(x) = -\frac{1}{2}x^2+\frac{1}{4}x^4$, causes the state to become metastable, where mass is transferred between the two aggregates until their masses equilibrate, which only happens in infinite time. As evidenced by the numerical example in Figure \ref{multistate}, where a four-aggregate and a five-aggregate state both exist as critical points for the same $\epsilon$ and $\nu$ values, it seems that diffusion does not guarantee a unique number of aggregates. It is clear that the four-aggregate state is preferred, as it has lower energy and requires fewer iterations of the fixed-point iterator. 

Using continuation on the diffusion parameter, as in Figure \ref{multistate}, suggests a method for computing the globally-minimizing configuration for each $\epsilon$. Both the four-aggregate and five-aggregate state are computed with final diffusion $\nu = 2^{-13}$, but the four-aggregate state is reached using continuation from initial diffusion $\nu_0 = 10\nu$, whereas the five-aggregate state uses $\nu_0= 2\nu$. The more energy-favourable state is reached from a larger starting $\nu_0$, which suggests that continuation from larger diffusion might be a mechanism for extracting the global minimizer. Ice crystallization provides a physical analogy: more imperfections form in ice crystals when water is frozen abruptly, indicating a non-energy-minimizing configuration, than when water is frozen slowly (see for instance \cite{kono2017effects}).
%
%
\begin{figure}[htp]
\centering
\begin{tabular}{c}
\includegraphics[trim={70 0 70 20},clip,width=0.8\textwidth]{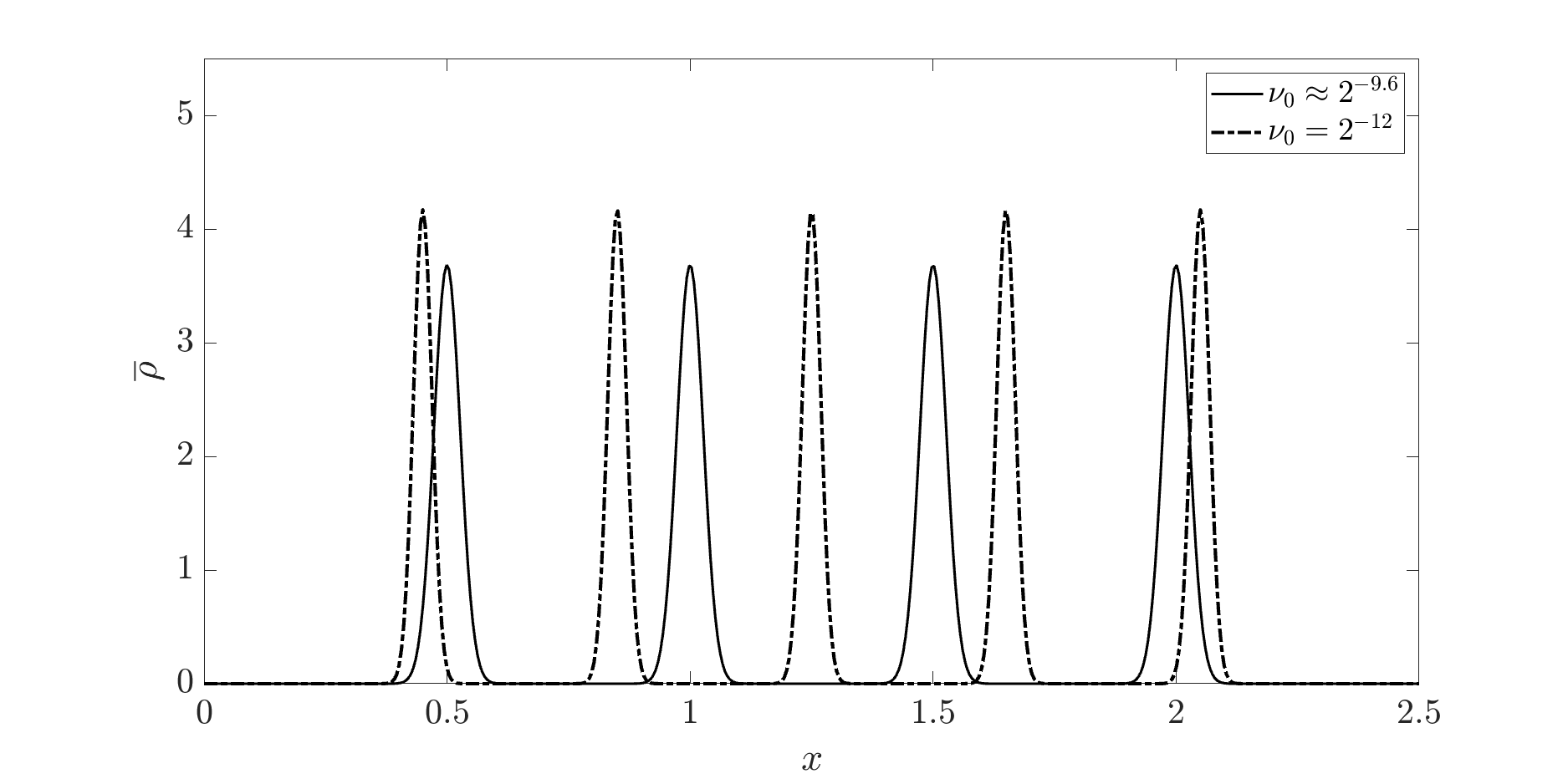} \\
\begin{tabular}{| c | c | c | c |}
\hline
$\nu_0$ & $\CalE^\nu$ & $\Lambda_\infty$ & Total Iterations \\ \hline
$10\nu$ & $-0.74841$ & $\sci{5.73}{7}$ & 26\\
$2\nu$ & $-0.74826$ & $\sci{1.19}{6}$ &  1149\\ \hline
\end{tabular}\\
\end{tabular}
\caption[Non-uniqueness of critical points for attractive-repulsive potentials]{Multiple critical points for $K_\epsilon$ with $\epsilon = 0.3$ and $\nu = 2^{-13}$. In each case, the fixed-point iterator was initiated at $\rho_0 = \frac{1}{L}\ind{[0,L]}$ and continuation was employed on the diffusion parameter. With initial diffusion $\nu_0 = 10\nu$, we arrive at a four-aggregate state, while for $\nu_0 = 2\nu$, a five-aggregate state emerges which has higher energy, higher $\Lambda_\infty$, and requires many more iterations of the fixed-point iterator, suggesting less stability.}
\label{multistate}
\end{figure}


\bigskip

{\large \bf Acknowledgements }  D.M. would like to thank his co-supervisor at Simon Fraser University, Ralf Wittenberg, for invaluable guidance during the writing of his Master's thesis. Ralf's thorough review of D.M.'s thesis directly led to a much better exposition of the material presented in this paper. The authors would also like to acknowledge Wittenberg for his insight during conversations related to the present article, and in particular, suggestions which led to the idea of the effective volume dimension for connecting diffusion-dominated spreading to domain geometry.  R.F. was supported by NSERC Discovery Grant PIN-341834 during this research.

\bibliographystyle{plain}
\bibliography{MeFe2019}

\end{document}

%% file: preamble.tex
\usepackage[utf8]{inputenc}
\usepackage{
amsmath,
amssymb,
amsthm,
appendix,
bbm,
caption,
color,
epsf,
enumitem,
float,
graphicx,
hyperref,
mathtools,
mathrsfs,
subcaption,
tikz,
titletoc,
url,
ulem,
xcolor
}

\hypersetup{
    colorlinks=true, 
    linktoc=all,     
    linkcolor=black,  
}

\numberwithin{equation}{section}

\newtheorem{thm}{Theorem}[section]
\newtheorem{lemm}{Lemma}[section]

\newtheorem{corr}{Corollary}[section]
\newtheorem{rmrk}{Remark}[section]

\newtheorem{assum}{Assumption}

\tikzset{every label/.style={font=\footnotesize,inner sep=1pt}}
\DeclareCaptionFormat{myformat}{#1#2#3\hrulefill}


%% file: notation.tex
\newcommand{\R}{{\mathbb R}}
\newcommand{\Rd}{\R^d}

\newcommand{\EE}{\mathbb{E}}
\newcommand{\RR}{\mathbb{R}}
\newcommand{\NN}{\mathbb{N}}

\newcommand{\ZZ}{\mathbb{Z}}

\newcommand{\CalC}{{\mathcal{C}}}

\newcommand{\CalB}{{\mathcal{B}}}
\newcommand{\CalE}{{\mathcal{E}}}
\newcommand{\CalF}{{\mathcal{F}}}

\newcommand{\CalK}{{\mathcal{K}}}

\newcommand{\CalP}{{\mathcal{P}}}
\newcommand{\CalS}{{\mathcal{S}}}
\newcommand{\CalV}{{\mathcal{V}}}
\newcommand{\CalW}{{\mathcal{W}}}

\newcommand{\munu}{\mu^{\nu}_t}
\newcommand{\rhonu}{\rho^\nu_t}

\newcommand{\critrho}{\overline{\rho}}

\newcommand{\gmin}{\overline{\rho}}

\newcommand{\fxpt}{{\tau}}
\newcommand{\fxptmap}{{T}}

\newcommand{\ballp}[3]{B_{{#1}}\left({#2},{#3}\right)} 
\newcommand{\com}{\CalC}

\newcommand{\rhofp}{\gmin_{_{FP}}}

\newcommand{\sci}[2]{{#1}\text{e}{-#2}}
\newcommand{\bcerr}{{E_0}}

\newcommand{\nrm}[1]{\left\Vert {#1} \right\Vert}

\newcommand{\dd}{\,d}
\newcommand{\ppt}{\frac{\partial}{\partial t}}

\newcommand{\supp}[1]{\text{supp}\left(#1\right)}
\newcommand{\dist}[2]{\text{dist}\left(#1,#2\right)}
\newcommand{\ind}[1]{\mathbbm{1}_{#1}}	
\newcommand{\eff}{{f}_{\scalebox{.9}{$\scriptscriptstyle D$}}}

\newcommand{\txt}[2]{\hspace{#1cm} \text{#2} \hspace{#1cm}} 

\renewcommand{\tilde}[1]{\widetilde{#1}}